\UseRawInputEncoding
\documentclass[leqno]{amsart}

\usepackage{stmaryrd,graphicx}
\usepackage{amssymb,mathrsfs,amsmath,amscd,amsthm,color}
\usepackage{float}
\usepackage{mathabx}
\usepackage[all,cmtip]{xy}
\DeclareMathAlphabet{\mathpzc}{OT1}{pzc}{m}{it}
\usepackage{amsfonts,latexsym,wasysym}

\renewcommand{\int}{\text{int}}

\newcommand{\sw}{\bigcurlyvee}

\newcommand{\ds}{\displaystyle}
\newcommand{\ui}{I}

\newcommand{\lb}{\langle}
\newcommand{\rb}{\rangle}

\newcommand{\scrr}{\mathscr{R}}
\newcommand{\scrs}{\mathscr{S}}

\newcommand{\bba}{\mathbb{A}}

\newcommand{\bbe}{\mathbb{E}}

\newcommand{\bbn}{\mathbb{N}}

\newcommand{\bbz}{\mathbb{Z}}

\newcommand{\simeqir}{\simeq_{ir}}

\newcommand{\llb}{\llbracket}
\newcommand{\rrb}{\rrbracket}
\newcommand{\im}{\operatorname{Im}}
\newcommand{\loopjn}{\Omega^{n}_{J}}
\newcommand{\loopn}{\Omega^{n}}
\newcommand{\mcw}{\mathcal{W}}

\newtheorem{theorem}{Theorem}[section]
\newtheorem{lemma}[theorem]{Lemma}
\newtheorem{proposition}[theorem]{Proposition}
\newtheorem{corollary}[theorem]{Corollary}
\theoremstyle{definition}\newtheorem{definition}[theorem]{Definition}
\newtheorem{example}[theorem]{Example}

\newtheorem{remark}[theorem]{Remark}

\newtheorem{conjecture}[theorem]{Conjecture}

\begin{document}
\title[Whitehead products and infinite sums]{Identities for Whitehead products\\and infinite sums}
\author[J. Brazas]{Jeremy Brazas}
\address{West Chester University\\ Department of Mathematics\\
West Chester, PA 19383, USA}
\email{jbrazas@wcupa.edu}


\date{\today}

\begin{abstract}
Whitehead products and natural infinite sums are prominent in the higher homotopy groups of the $n$-dimensional infinite earring space $\mathbb{E}_n$ and other locally complicated Peano continua. In this paper, we derive general identities for how these operations interact with each other. As an application, we consider a shrinking wedge $X=\bigcurlyvee_{j\in\mathbb{N}}X_j$ of finite $(n-1)$-connected CW-complexes and compute the infinite-sum closure $\mathcal{W}_{2n-1}(X)$ of the set of Whitehead products $[\alpha,\beta]$ in $\pi_{2n-1}\left(X\right)$ where $\alpha,\beta\in\pi_n(X)$ are represented in respective sub-wedges that meet only at the basepoint. In particular, we show that $\mathcal{W}_{2n-1}(X)$ is canonically isomorphic to $\prod_{j=1}^{\infty}\left(\pi_{n}(X_j)\otimes \prod_{k>j}\pi_n(X_k)\right)$. The insight provided by this computation motivates a conjecture about the isomorphism type of the elusive groups $\pi_{2n-1}(\mathbb{E}_n)$, $n\geq 2$. 
\end{abstract}

\maketitle

\section{Introduction}

Following Katsuya Eda's homotopy classification of one-dimensional Peano continua \cite{Edaonedim}, there has been increased interest in the higher homotopy theory of Peano continua. Perhaps the most fundamental problem in this arena is to characterize the homotopy groups of the $n$-dimensional earring space $\bbe_n$ in terms of the homotopy groups of spheres as an infinitary extension of Hilton's Theorem \cite{HiltonSpheres}. Even a full description of $\pi_3(\bbe_2)$ for the $2$-dimensional earring $\bbe_2$ remains elusive. This paper is intended to provide some progress toward establishing such descriptions. Combining the insights of the current paper with methods developed in the author's recent paper \cite{BrazSequentialnconn}, we are motivated to make the first conjecture about the isomorphism type of $\pi_{2n-1}(\bbe_n)$. 

Homotopy groups of spaces like $\bbe_n$ cannot be computed using standard methods in homotopy theory because these groups come equipped with a non-trivial infinite sum operation $\left[\sum_{j=1}^{\infty}f_j\right]$, which may be formed when the images of the based maps $f_j:S^n\to X$ converge to the basepoint. Such infinite sums extend the usual sum operation in homotopy groups, behave like absolutely convergent series of complex numbers in the sense that they are infinitely commutative \cite{Brazasncubeshuffle}, and sit at the forefront of the few existing computations that have been completely established \cite{EKR09,EK00higher,EKcones,Kawamurasuspensions}. There is substantial evidence, e.g. \cite{ABpacific,EKR07,EKR09nonaspher,Felt,KarimovRepovs}, suggesting that shape theoretic methods are insufficient for making characterizing homotopy groups or homotopy types of Peano continua. As new methods of computation continue to be developed in the infinitary setting, there is a need to understand how infinite sums and classical homotopy operations interact with each other. In this paper, we establish identities for how to deal with various combinations of infinite sums and Whitehead products, e.g. when infinite sums occur within Whitehead products and also when Whitehead products occur as the summands of an infinite sum.

We use commutator bracket notation for both maps and homotopy classes. On iterated loop spaces, there is a Whitehead bracket operation $\Omega^p(X)\times \Omega^{q}(X)\to \Omega^{p+q-1}(X)$, $(f,g)\mapsto [f,g]$ and this descends to a bracket operation $\pi_{p}(X)\times \pi_{q}(X)\to \pi_{p+q-1}(X)$, $(\alpha,\beta)\mapsto [\alpha,\beta]$ on homotopy classes. Whitehead products satisfy a variety of important identities such as bilinearity, graded symmetry, and a graded Jacoby identity (see Remark \ref{relations}). In Section \ref{sectionidentity}, we observe that these classical identities may be realized by deformations occurring entirely within the domain and may therefore be applied infinitely many times within a single homotopy. Due to topological restrictions, one should not expect an infinite analogue of bilinearity that allows one to ``pull out" an infinite sum from within a Whitehead product $\left[\sum_{j=1}^{\infty}f_j,g\right]$ by a continuous deformation. Indeed, the map $\left[\sum_{j=1}^{\infty}f_j,g\right]$ is defined if the maps $f_j:S^n\to X$ shrink toward the basepoint but the infinite sum $\sum_{j=1}^{\infty}[f_j,g]$ is not a continuous map unless $g$ is constant.  A culminating identity established in Section \ref{sectionidentity} is Theorem \ref{mainidentitylemma} that gives a useful expansion of the form $\sum_{i=1}^{\infty}\left[\sum_{j=1}^{\infty}f_{i,j},\sum_{j=1}^{\infty}g_{i,j}\right]$ when $f_{i,j},g_{i,j}\in\Omega^n(X)$.

One can readily visualize the deformations we construct in dimension $p=q=2$. The domain of a Whitehead product $[\alpha,\beta]\in \pi_3(X,x_0)$ of $\alpha,\beta\in \pi_2(X,x_0)$ may be represented by a Hopf link of solid tori in the interior of $I^3$. A mapping $[f,g]$ that represents $[\alpha,\beta]$ is defined by (1) sending cross-sectional disks $D$ of the two tori to $X$ by relative maps $f,g:(D,\partial D)\to (X,x_0)$ that represent $\alpha,\beta$ respectively and (2) mapping all points outside the interiors of the tori to $x_0$. The domain of an infinite sum map $\sum_{j=1}^{\infty}[f_j,g_j]$ can be decomposed as an infinite sequence of thickened Hopf links placed side-by-side and so that each link is mapped by $[f_j,g_j]$ as in the singular case. Note that these pairwise-disjoint Hopf links do not necessarily have to shrink toward a point in the domain. However, their images in $X$ must shrink toward the basepoint as $j\to\infty$. The domain of a Whitehead product map $\left[\sum_{j=1}^{\infty}f_j,g\right]$ is represented by a single torus $T$ (the domain for $g$) simultaneously linked with infinitely many solid tori $S_j$ for $f_j$. The tori $S_j$ will shrink in thickness and converge to a circle, which is linked with $T$ and mapped to the basepoint (see the left image in Figure \ref{tori}). Lastly, $\left[\sum_{j=1}^{\infty}f_j,\sum_{j=1}^{\infty}g_j\right]$ is represented by infinitely many tori $T_j$ simultaneously linked with infinitely many tori $S_j$ where both sequences of tori converge to circles (see the right image in Figure \ref{tori}). From this visual perspective, the identities we derive carry out infinite sequences of  ``pinching," ``joining," and ``rearrangement" deformations to families of linked manifolds in the domains of maps on spheres. 
\begin{figure}[h]
\centering \includegraphics[height=1.8in]{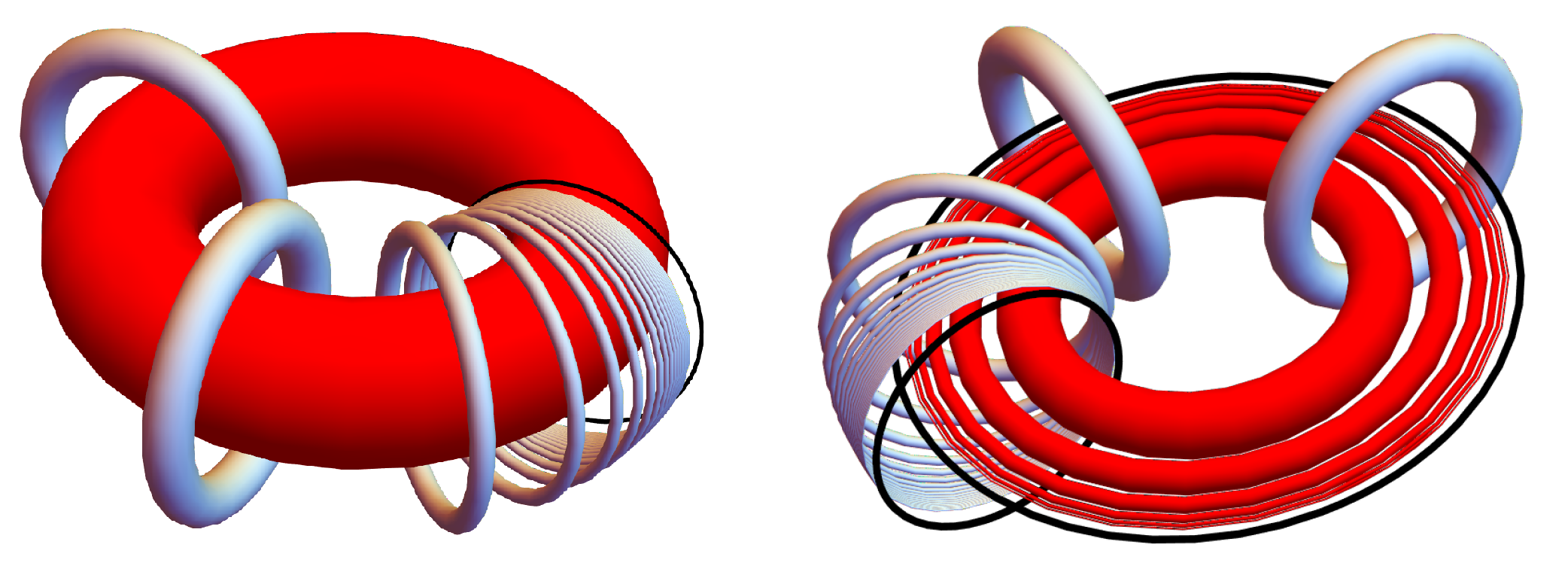}
\caption{\label{tori}Representation of the domain of Whitehead products $\left[\sum_{j=1}^{\infty}f_j,g\right]$ (left) and $\left[\sum_{j=1}^{\infty}f_j,\sum_{j=1}^{\infty}g_j\right]$ (right).}
\end{figure}

To illustrate the utility of our identities, we analyze the first situation in which they are applicable, namely Whitehead products $[\alpha,\beta]\in\pi_{2n-1}(X)$ for $\alpha,\beta\in\pi_{n}(X)$, $n\geq 2$. To provide some context for our main theorem, we recall the finite scenario. If $m\in\bbn$ and $X=\bigvee_{j=1}^{m}X_j$ is a wedge of $(n-1)$-connected CW-complexes, standard methods in homotopy theory \cite{WhiteheadEOH} may be used to compute $\pi_{2n-1}(X)$ in terms of the homotopy groups of the wedge-summands $X_j$. In particular, the long-exact sequence of the pair $(\prod_{j=1}^{m}X_j,X)$ gives a splitting of $\pi_{2n-1}(X)$ into the direct sum of the groups $\pi_{2n-1}(\prod_{j=1}^{m}X_j)\cong\bigoplus_{j=1}^{m}\pi_{2n-1}(X_j)$ and $\pi_{2n}(\prod_{j=1}^{m}X_j,X)\cong \bigoplus_{j=1}^{m-1}\left(\pi_{n}(X_j)\otimes \bigoplus_{k=j+1}^{m}\pi_{n}(X_k)\right)$. The latter sum may be regarded as the direct sum of the groups $\pi_{n}(X_j)\otimes \pi_{n}(X_k)$ indexed over two-element subsets $\{j,k\}$ of $\{1,2,\dots,m\}$. Moreover, the elements of this sum may be characterized by the fact that the resulting injection $\pi_{n}(X_j)\otimes \pi_{n}(X_k)\to \pi_{2n-1}(X)$ is precisely the Whitehead bracket $\alpha\otimes\beta\mapsto [\alpha,\beta]$. 

As an infinite extension of the finite wedges in the previous paragraph, we consider a countably infinite ``shrinking" wedge $X=\sw_{j\in\bbn}X_j$ of $(n-1)$-connected CW-complexes in which every neighborhood of the basepoint contains all but finitely many of the wedge-summand spaces $X_j$ (see Definition \ref{swdef}). For example, $\bbe_n=\sw_{j\in\bbn}S^n$ is the shrinking wedge of $n$-spheres and satisfies the connectedness hypotheses \cite{EK00higher}. We will write $X_{>k}$ to mean the sub-shrinking wedge $\sw_{j=k+1}^{\infty}X_j$ consisting of all but the first $k$ wedge-summands. Infinite commutativity and the long exact sequence of the pair $(\prod_{j=1}^{\infty}X_j,X)$ give an analogous splitting of $\pi_{2n-1}(X)$ as the sum of $\pi_{2n-1}(\prod_{j=1}^{\infty}X_j)\cong \prod_{j=1}^{\infty}\pi_{2n-1}(X_j)$ and $\pi_{2n}(\prod_{j=1}^{\infty}X_j,X)\cong \ker(\sigma_{\#})$ where $\sigma:X\to \prod_{j=1}^{\infty}X_j$ is the inclusion map \cite{Kawamurasuspensions}. Hence, to compute $\pi_{2n-1}(X)$ in terms of the homotopy groups of the wedge-summand spaces $X_j$, we may focus our attention to $\ker(\sigma_{\#})$.

Recall from the finite-wedge situation that the elements of the relative homotopy group are given by Whitehead products $[\alpha,\beta]$ where $\alpha$ and $\beta$ come from \textit{distinct} wedge-summands. We require an infinite analogue of this idea. A pair $(\alpha,\beta)$ of elements of $\pi_{n}(X)$ are said to be \textit{wedge-disjoint} if there are disjoint subsets $F,G\subseteq \bbn$ such that $\alpha\in \pi_n(\bigcup_{j\in F}X_j)$ and $\beta\in \pi_n(\bigcup_{j\in G}X_j)$. We define a subgroup $\mathcal{W}_{2n-1}(X)$ of $\pi_{2n-1}(X)$ that turns out to be the smallest subgroup of $\pi_{2n-1}(X)$ that (1) is closed under infinite sums and (2) contains all Whitehead products $[\alpha,\beta]$ such that $\alpha,\beta\in\pi_n(X)$ and the pair $(\alpha,\beta)$ is wedge-disjoint. After showing that $\mathcal{W}_{2n-1}(X)$ is a subgroup of $\ker(\sigma_{\#})$ (Corollary \ref{inkercor}), we make the following computation.

\begin{theorem}\label{maintheorem}
If $n\geq 2$ and $X=\sw_{j\in\bbn}X_j$ is a shrinking wedge of finite $(n-1)$-connected CW-complexes, then the canonical homomorphism \[\Phi:\prod_{j=1}^{\infty}\left(\pi_{n}(X_j)\otimes \pi_{n}(X_{>j})\right)\to \pi_{2n-1}(X)\]given by $\Phi\left(\left(\sum_{i=1}^{r_j}[f_{j,i}]\otimes[g_{j,i}]\right)_{j\in\bbn}\right)= \left[\sum_{j=1}^{\infty}\sum_{i=1}^{r_j}[f_{j,i},g_{j,i}]\right]$ is injective and has image equal to $\mcw_{2n-1}(X)$. Moreover, $\mcw_{2n-1}(X)\cong \prod_{j=1}^{\infty}\left(\pi_{n}(X_j)\otimes \prod_{k>j}\pi_n(X_k)\right)$.
\end{theorem}

Proving that $\im(\Phi)=\mcw_{2n-1}(X)$ requires a combination of several identities from Sections \ref{sectionprelim} and \ref{sectionidentity}. Proving that $\Phi$ is injective requires a combination of techniques from shape theory \cite{MS82} and a non-trivial computation of Eda-Kawamura \cite{EK00higher}. The hypothesis that the CW-complexes $X_j$ are finite appears to be required for both parts of the proof since the groups $\pi_n(X_j)$ must be finitely generated in order to apply a special situation where tensor products preserve certain infinite direct products (Lemma \ref{tensorlemma}).

\begin{example}
In the case of the $n$-dimensional earring space $\bbe_n=\sw_{k\in\bbn}S^n$, Theorem \ref{maintheorem} gives that the subgroup $\ds\mcw_{2n-1}(\bbe_n)$ of $\pi_{2n-1}(\bbe_n)$ is canonically isomorphic to \[\prod_{j=1}^{\infty}\left(\pi_{n}(S^n)\otimes \prod_{k>j}\pi_{n}(S^n)\right)\cong \prod_{j=1}^{\infty}\left(\bbz\otimes \prod_{k>j}\bbz\right)\cong \prod_{j=1}^{\infty} \prod_{k>j}\bbz,\] which is abstractly isomorphic to the Baer-Specker group $\bbz^{\bbn}$. However, the representation $\prod_{j=1}^{\infty} \prod_{k>j}\bbz$ is most insightful since it illustrates a ``standard form" for elements of $\mcw_{2n-1}(\bbe_n)$. Namely, if $\ell_j:S^n\to \bbe_n$ is the inclusion of the $j$-th sphere, then every element of $\mcw_{2n-1}(\bbe_n)$ may be uniquely represented by a map of the form $\sum_{j=1}^{\infty}\left[\ell_j,\sum_{k>j}m_{j,k}\ell_k\right]$ for integers $m_{j,k}$ indexed over pairs $(j,k)$ where $k>j\geq 1$.
\end{example}

The techniques developed in \cite{BrazSequentialnconn} provide a method for deforming an arbitrary map $f:S^m\to \sw_{j\in\bbn}X_j$ into one with much simpler structure. When combined with the insights from the current paper, it appears likely that there are, in fact, no unaccounted-for elements of $\pi_{2n-1}(\bbe_n)$. Indeed, it is not difficult to visualize the required deformations in the case of $\pi_{3}(\bbe_2)$. However, new methods will be necessary in order for this intuition to be formalized. Hence, the author is prepared to make the following conjecture.

\begin{conjecture}
If $n\geq 2$ and $X=\sw_{k\in\bbn}X_k$ is a shrinking wedge of finite $(n-1)$-connected CW-complexes and $\sigma:X\to \prod_{j=1}^{\infty}X_j$ is the inclusion map, then $\mcw_{2n-1}(X)=\ker(\sigma_{\#}:\pi_{2n-1}(X)\to \pi_{2n-1}(\prod_{j=1}^{\infty}X_j))$.
\end{conjecture}

As mentioned above, $\sigma_{\#}:\pi_{2n-1}(X)\to \pi_{2n-1}(\prod_{j=1}^{\infty}X_j)$ is a split epimorphism. Thus, a positive answer to this conjecture would establish that $\pi_{2n-1}(X)$ is isomorphic to \[\prod_{j=1}^{\infty}\pi_{2n-1}(X_j)\oplus\mcw_{2n-1}(X)\cong \prod_{j=1}^{\infty}\pi_{2n-1}(X_j)\oplus \prod_{j=1}^{\infty}\left(\pi_{n}(X_j)\otimes \prod_{k>j}\pi_n(X_k)\right).\]In the case of the $n$-dimensional earring space, the conjecture would imply that $\pi_{2n-1}(\bbe_n)\cong \pi_{2n-1}(S^n)^{\bbn}\oplus \bbz^{\bbn}$.

\section{Preliminaries and notation}\label{sectionprelim}

Generally, all (based) topological spaces will be assumed to be Hausdorff and the term \textit{map} will mean a (based) continuous function. Throughout, $\ui$ will denote the closed unit interval, $S^n$ will denote the unit $n$-sphere with basepoint $(1,0,0,\dots,0)$. For a based topological space $(X,x_0)$, we write $\Omega^{n}(X,x_0)$ to denote the space of relative maps $(I^n,\partial I^n)\to (X,x_0)$ and $\pi_n(X,x_0)=\{[f]\mid f\in \Omega^n(X,x_0)\}$ to denote the $n$-th homotopy group. When the basepoint is clear from context, we may simplify this notation to $\Omega^n(X)$ and $\pi_n(X)$. By fixing a based homeomorphism $I^n/\partial I^n\cong S^n$, we identify $\Omega^{n}(X)$ with the space of based maps $S^n\to X$. We only consider higher homotopy groups in this paper and so we generally assume $n\geq 2$. 

For $f\in \Omega^{n}(X)$, we let $-f$ denote the reverse map given by $-f(t_1,t_2,\dots,t_n)=f(1-t_1,t_2,\dots,t_n)$ so that $-[f]=[-f]$. We write $f_1+f_2+\dots+f_m$ for the usual first-coordinate concatenation of maps $f_i\in \Omega^n(X)$ so that $[\sum_{i=1}^{m}f_i]=\sum_{i=1}^{m}[f_i]$ in $\pi_n(X)$. If $m$ is a positive integer, then $m f$ denotes the $m$-fold sum $f+f+\cdots +f$, $0f$ denotes the constant map, and if $m$ is a negative integer, then $m f$ denotes the $(-m)$-fold sum $(-f)+(-f)+\cdots+(-f)$.

For an ordered pair of natural numbers $(p,q)\in\bbn^2$, we view $S^p\vee S^q$ naturally as a subspace of $S^p\times S^q$. Let $a_{p,q}:S^{p+q-1}\to S^p\vee S^q$ be a fixed attaching map of the unique $(p+q)$-cell in $S^p\times S^q$. If $f:S^p\to X$ and $g:S^q\to X$ are based maps, we refer to the composition $[ f,g]=(f\vee g)\circ a_{p,q}:S^{p+q-1}\to X$ as the \textit{Whitehead product of the maps $f$ and $g$}. The bracket operation $\Omega^p(X)\times \Omega^q(X)\to \Omega^{p+q-1}(X)$ on maps descends to a well-defined bracket operation $\pi_p(X)\times \pi_q(X)\to \pi_{p+q-1}(X)$, $(\alpha,\beta)\mapsto [\alpha,\beta]$ on homotopy classes, which we also refer to as the Whitehead product.

\begin{remark}\label{relations}
It is well-known that Whitehead products satisfy a variety of relations. For based space $X$, integers $p,q,r\geq 2$, and elements $a,a'\in\pi_p(X)$, $b,b'\in \pi_q(X)$, and $c\in \pi_r(X)$, the following relations hold, c.f. \cite[Chp. X.7]{WhiteheadEOH}.
\begin{enumerate}
\item (identity) $[a,0]=[0,a]=0$,
\item (bilinearity) $[a+a',b]=[a,b]+[a',b]$ and $[a,b+b']=[a,b]+[a,b']$,
\item (graded symmetry) $[a,b]=(-1)^{pq}[b,a]$,
\item (graded Jacobi identity) \[(-1)^{pr}[[a,b],c]+(-1)^{pq}[[b,c],a]+(-1)^{rq}[[c,a],b]=0.\]
\end{enumerate}
\end{remark}

The bilinearity of Whitehead products ensures that, for based space $X$, there is a natural homomorphism $\pi_p(X)\otimes \pi_q(X)\to \pi_{p+q-1}(X)$ that maps a simple tensor $\alpha\otimes \beta$ to the Whitehead product $[\alpha,\beta]$.

\begin{remark}\label{inducedmapremark}
If $h:X\to Y$ is a map, $f\in\Omega^p(X)$, and $g\in \Omega^q(X)$, then $h\circ[ f,g]=[ h\circ f,h\circ g]$. Thus, if $h_{\#}:\pi_{p+q-1}(X)\to \pi_{p+q-1}(Y)$ is the induced homomorphism, then $h_{\#}([\alpha,\beta])=[h_{\#}(\alpha),h_{\#}(\beta)]$ on homotopy classes.
\end{remark}

\subsection{Infinite sums in higher homotopy groups}

\begin{definition}
Let $J$ be a countably infinite set and $\{f_j\}_{j\in J}$ be a $J$-indexed collection of maps $f_j:W_j\to X$ of topological spaces. 
\begin{enumerate}
\item We say that the collection $\{f_j\}_{j\in J}$ \textit{clusters at} $x_0\in X$ if for every neighborhood $U$ of $x_0$, we have $\im(f_j)\subseteq U$ for all but finitely many $j\in J$. Moreover, if $J=\bbn$, we say that $\{f_j\}_{j\in \bbn}$ \textit{converges to} $x_0$.
\item Let $\Omega^{n}_{J}(X,x_0)$ denote the set of $J$-indexed collections $\{f_j\}_{j\in J}$ where $f_j\in \Omega^n(X,x_0)$ for all $j\in J$ and such that $\{f_j\}_{j\in J}$ clusters at $x_0$.
\item A sequence $\{\alpha_j\}_{j\in\bbn}$ where $\alpha_j\in \pi_n(X,x_0)$ for all $j\in \bbn$ is said to be \textit{summable} if there exists $\{f_j\}_{j\in \bbn}\in\loopn_{\bbn}(X,x_0)$ such that $\alpha_j=[f_j]$ for all $j\in\bbn$.
\end{enumerate}
\end{definition}

\begin{definition}
The \textit{infinite concatenation} (or \textit{infinite sum} if $n\geq 2$) of a sequence $\{f_j\}_{j\in \bbn}\in \loopn_{\bbn}(X,x_0)$ is the map $\sum_{j=1}^{\infty}f_j\in \loopn(X,x_0)$ defined as $f_j$ on $\left[\frac{j-1}{j},\frac{j}{j+1}\right]\times I^{n-1}$ for all $j\in\bbn$ and which maps $\{1\}\times I^{n-1}$ to $x_0$.
\end{definition}

Note that $\im(\sum_{j=1}^{\infty}f_j)=\bigcup_{j\in\bbn}\im(f_j)$ and if $h:X\to Y$ is a based map, then $h\circ \sum_{j=1}^{\infty}f_j=\sum_{j=1}^{\infty}h\circ f_j$.

\begin{remark}
We avoid writing infinite sums of homotopy classes, i.e. $\sum_{j=1}^{\infty}[f_j]$, since there exist spaces admitting sequences $\{f_j\}_{j\in\bbn},\{g_j\}_{j\in\bbn}\in \Omega^{n}_{\bbn}(X,x_0)$ such that $[f_j]=[g_j]$ for all $j\in\bbn$ but such that $\sum_{j=1}^{\infty}[f_j]\neq \sum_{j=1}^{\infty}[g_j]$ (see Example \ref{archipelagoexample}). Hence, the function $\Omega_{n}^{\bbn}(X,x_0)\mapsto \pi_n(X,x_0)$ given by $\{f_j\}_{j\in\bbn}\mapsto \left[\sum_{j=1}^{\infty}f_j\right]$ is well-defined, but the corresponding operation $\{\alpha_j\}_{j\in\bbn}\mapsto \sum_{j=1}^{\infty}\alpha_j$ on summable sequences of homotopy classes may not be well-defined. 

The identities in this paper will therefore be stated in terms of homotopies of representing maps. For example, one could make sense of an identity  $\sum_{j=1}^{\infty}[\alpha_j,\beta_j]=(-1)^{pq}\sum_{j=1}^{\infty}[\alpha_j,\beta_j]$ for summable sequences $\{\alpha_j\}_{j\in\bbn}$ and $\{\beta_j\}_{j\in\bbn}$ of homotopy classes. However, the value of both the left and right sides of this equality depends on a choice of a sequence of representing maps. Thus, we represent this identity more carefully by saying that there is a homotopy $\sum_{j=1}^{\infty}[f_j,g_j]\simeq(-1)^{pq}\sum_{j=1}^{\infty}[f_j,g_j]$ for any sequences $\{f_j\}_{j\in\bbn}$ and $\{g_j\}_{j\in\bbn}$ in $\loopn_{\bbn}(X,x_0)$. Moreover, we will often keep track of the images of the homotopies required for such a relation. This leads us to the next definition.
\end{remark}

\begin{definition}
Let $f,g:X\to Y$ be based maps of topological spaces. A based homotopy $H:X\times \ui\to Y$ from $f$ to $g$ is an \textit{image-relative homotopy} if $\im(H)=\im(f)\cup \im(g)$. When such a homotopy exists, we say the maps $f$ and $g$ are \textit{image-relative homotopic} and we write $f\simeqir g$.
\end{definition}

Image-relative homotopy $\simeqir$ is a reflexive and symmetric relation on the set of based maps $X\to Y$ that is much finer than the usual based homotopy equivalence relation $\simeq$. Note that if $c,f,f',g,g'\in \Omega^n(X)$ where $c$ is the constant map, then
\begin{enumerate}
\item $f+c\simeqir f\simeqir c+f$,
\item $f+(-f)\simeqir c\simeqir (-f)+f$,
\item if $f\simeq_{ir}f'$ and $g\simeq_{ir}g'$, then $f+g\simeq_{ir}f'+g'$.
\end{enumerate}

\begin{remark}\label{weaktransitiveremark}
Although, $\simeqir$ is not transitive, if $f\simeqir g$ and $g\simeqir h$ where $\im(g)\subseteq \im(f)\cup \im(h)$, then $f\simeqir h$.
\end{remark}

\subsection{Infinite compositions of homotopies}

\begin{definition}
Let $\{f_j\}_{j\in\bbn}$ and $\{g_j\}_{j\in\bbn}$ be elements of $\loopn_{\bbn}(X,x_0)$. A \textit{sequential homotopy} from $\{f_j\}_{j\in\bbn}$ to $\{g_j\}_{j\in\bbn}$ is a sequence $\{H_j\}_{j\in\bbn}$ of maps $H_j:S^n\times \ui\to X$ that converges to $x_0$ and where $H_j$ is a based homotopy from $f_j$ to $g_j$ for each $j\in\bbn$. If such a sequence exists, we say that $\{f_j\}_{j\in\bbn}$ and $\{g_j\}_{j\in\bbn}$ are \textit{sequentially homotopic} and we write $\{f_j\}_{j\in\bbn}\simeq\{g_j\}_{j\in\bbn}$. If $g_j$ is constant at $x_0$ for all $j\in\bbn$, we say $\{f_j\}_{j\in\bbn}$ is sequentially \textit{sequentially null-homotopic} and call $\{H_j\}_{j\in\bbn}$ a \textit{sequential null-homotopy}.
\end{definition}

Sequential homotopy is an equivalence relation on $\Omega^{n}_{\bbn}(X,x_0)$. We refer to the equivalence class of $\{f_j\}_{j\in\bbn}$ as its \textit{sequential homotopy class}.

\begin{remark}[Infinite horizontal concatenation]\label{infconcatenationremark}
If $\{H_j\}_{j\in\bbn}$ is a sequential homotopy from $\{f_j\}_{j\in\bbn}$ to $\{g_j\}_{j\in\bbn}$ (where each is an element of $\loopn_{\bbn}(X,x_0)$), then we may define a homotopy $H:I^n\times \ui\to X$ to be $H_j$ on $\left[\frac{j-1}{j},\frac{j}{j+1}\right]\times I^{n-1}\times I$ and so that $H(\{1\}\times I^{n-1}\times I)=x_0$. The map $H$ is continuous and defines a based homotopy $\sum_{j=1}^{\infty}H_j$ from $\sum_{j=1}^{\infty}f_j$ to $\sum_{j=1}^{\infty}g_j$, which we refer to as the \textit{infinite horizontal concatenation} of the sequence $\{H_j\}_{j\in\bbn}$.
\end{remark}

The infinite horizontal concatenation construction leads to the following, which we will often use without reference.

\begin{proposition}\label{irsumsprop}
Suppose $\{f_j\}_{j\in \bbn}$ and $\{g_j\}_{j\in \bbn}$ are elements of $\loopn_{\bbn}(X,x_0)$.
\begin{enumerate}
\item if $\{f_j\}_{j\in\bbn}\simeq\{g_j\}_{j\in\bbn}$, then $\sum_{j=1}^{\infty}f_j\simeq\sum_{j=1}^{\infty}g_j$,
\item if $\{f_j\}_{j\in\bbn}$ is sequentially null-homotopic, then $\sum_{j=1}^{\infty}f_j$ is null-homotopic,
\item if $f_j\simeqir g_j$ for all $j\in\bbn$, then $\{f_j\}_{j\in\bbn}$ and $\{g_j\}_{j\in\bbn}$ are sequentially homotopic and $\sum_{j=1}^{\infty}f_j\simeqir\sum_{j=1}^{\infty}g_j$.
\end{enumerate}
\end{proposition}

\begin{proof}
(1) follows from the infinite horizontal concatenation construction and (2) is a special case of (1). For (3), note if $f_j\simeqir g_j$ for all $j\in\bbn$, then we may choose each based homotopy $H_j$ from $f_j$ to $g_j$ so that $\im(H_j)=\im(f_j)\cup \im(g_j)$. Since $\{f_j\}_{j\in\bbn}$ and $\{g_j\}_{j\in\bbn}$ both converge to $x_0$ and $\im(H_j)=\im(f_j)\cup \im(g_j)$ for all $j$, the sequence $\{H_j\}_{j\in\bbn}$ converges to $x_0$. Hence, $\{f_j\}_{j\in\bbn}$ and $\{g_j\}_{j\in\bbn}$ are sequentially homotopic. Finally, since $\im(H)=\bigcup_{j\in\bbn}\im(H_j)=\bigcup_{j\in\bbn}\im(f_j)\cup \im(g_j)=\im(\sum_{j=1}^{\infty}f_j)\cup \im(\sum_{j=1}^{\infty}g_j)$, the two sums are image-relative homotopic.
\end{proof}

\begin{remark}[Infinite vertical concatenation]\label{infverticalremark}
We will also be confronted by a situation which involves a sequence $f_j\simeq g_j+f_{j+1}$, $j\in\bbn$ of recursive homotopy relations in $\Omega^n(X,x_0)$ realized by a sequential homotopy $\{L_j\}_{j\in\bbn}$ from $\{f_j\}_{j\in\bbn}$ to $\{g_j+f_{j+1}\}_{j\in\bbn}$. We define a homotopy $G$ from $f_1$ to $\sum_{j=1}^{\infty}g_j$ piecewise on the sets $A_j=I^{n}\times \left[\frac{j-1}{j},\frac{j}{j+1}\right]$ and so that $G({\bf s},1)= \left(_{j=1}^{\infty}\right)({\bf s})$. In particular, on $A_1$, $G$ is defined to be the homotopy $L_1$ from $f_{1}$ to $g_1+ f_2$. On $A_k$, $k\geq 2$, $G$ is defined to be the homotopy from $g_1+ (g_2+ (g_3(\cdots (g_{k-1}+ f_{k}))))$ to $ g_1+ (g_2+ (g_3(\cdots g_{k-1}+ ( g_k+ f_{k+1} ))))$, which is the constant homotopy on the domains of $g_1,\dots,g_{k-1}$ and which is $L_{k}$ on the domain of $f_k$. Composing $G$ with a reparameterization homotopy $g_1+ (g_2+ (g_3+\cdots ))\simeq \sum_{j=1}^{\infty}g_j$ (which can be defined explicitly or using Remark \ref{infassocremark} below), we obtain $f_1\simeq\sum_{j=1}^{\infty}g_j$ by a homotopy $G'$ with $\im(G')=\bigcup_{k=1}^{\infty}\im(L_k)$. We refer to $G'$ as an \textit{infinite vertical concatenation of homotopies}. Note that if each $L_j$ is an image-relative homotopy, then so is $G$.
\end{remark}

\subsection{Infinite commutativity identities}

We recall several ``infinite-commutativity" identities, all of which are derived from an infinite higher-cube-shuffling result that first appeared in \cite{EK00higher} and was reformulated for broader application in \cite{Brazasncubeshuffle}. The term \textit{$n$-cube} will refer to sets of the form $\prod_{i=1}^{n}[a_i,b_i]$ in $\ui^n$. An \textit{$n$-domain} is a countable collection of $n$-cubes in $\ui^n$ with pairwise-disjoint interiors. If $\scrr$ is an $n$-domain and $\{f_R\}_{R\in\scrr}$ is an $\scrr$-indexed collection of maps $f_R\in\Omega^n(X,x_0)$ that clusters at $x_0$, then the \textit{$\scrr$-concatenation} of $\{f_R\}_{R\in\scrr}$ is the map $\prod_{R\in\scrr}f_R$ which is defined to be $f_R$ on $R$ and which maps $\ui^n\backslash\bigcup_{R\in \scrr}\int(R)$ to $x_0$.

\begin{theorem}\cite[Theorem 1.1]{Brazasncubeshuffle}\label{cubeshuffletheorem}
Let $\scrr$ and $\scrs$ be $n$-domains and $\phi:\scrs\to\scrr$ be a bijection. If $\{f_R\}_{R\in\scrr}$ and $\{g_S\}_{S\in\scrs}$ are collections of elements in $\Omega^n(X,x_0)$ that cluster at $x_0$ and such that $f_R=g_{S}$ whenever $\phi(S)=R$, then we have $\prod_{R\in\scrr}f_R\simeqir \prod_{S\in\scrs}g_S$.
\end{theorem}

A first application of Theorem \ref{cubeshuffletheorem} tells us that infinite sums (up to image-relative homotopy) behave like absolutely convergent series in that their homotopy classes do not depend on the order in which their summands are added.

\begin{corollary}\label{infcomm1}\cite{Brazasncubeshuffle}
If $\{f_j\}_{j\in\bbn}$ is an element of $\Omega^{n}_{\bbn}(X,x_0)$ and $\phi:\bbn\to\bbn$ is any bijection, then $\sum_{j=1}^{\infty}f_j\simeqir \sum_{j=1}^{\infty}f_{\phi(j)}$.
\end{corollary}

We also wish to make sense of a sum $\sum_{j\in J}f_j$ when $J$ is a more general linearly ordered set. For a given countable linearly ordered set $(J,\leq)$, we fix a collection $\{(a_j,b_j)\}_{j\in J}$ of pairwise-disjoint open intervals in $\ui$ such that $a_j<a_{j'}$ whenever $j<j'$ in $J$. If $\{f_j\}_{j\in J}\in \loopn_{J}(X,x_0)$, we define $\sum_{j\in J}f_j$ to be the map $f:(\ui^n,\partial \ui^n)\to (X,x_0)$, which is $f_j$ on $[a_j,b_j]\times\ui^{n-1}$ and which maps $(\ui\backslash \bigcup_{j\in J}[a_j,b_j])\times\ui^{n-1}$ to $x_0$.

\begin{corollary}\label{infcomm2}\cite{Brazasncubeshuffle}
Let $J$ be a countable linearly ordered set. If $\{f_j\}_{j\in J}$ is an element of $\Omega^{n}_{J}(X,x_0)$ and $\phi:\bbn\to J$ is any bijection, then $\sum_{j\in J}f_j\simeqir \sum_{i=1}^{\infty}f_{\phi(i)}$.
\end{corollary}

\begin{remark}[Infinite Associativity]\label{infassocremark}
Theorem \ref{cubeshuffletheorem} ensures that the choice of the intervals $(a_j,b_j)$ in the definition of $\sum_{j\in J}f_j$ does not change $\sum_{j\in J}f_j$ up to image-relative homotopy. This implies ``infinite associativity for infinite sums." For example, the following maps are all homotopic to each other by image-relative homotopies. We refer to such homotopies as \textit{reparameterization homotopies}.
\begin{enumerate}
\item $\sum_{j=1}^{\infty}f_j$,
\item $\sum_{k\in\bbn}(f_{2k-1}+f_{2k})=(f_1+f_2)+(f_3+f_4)+(f_5+f_6)+\cdots$,
\item $f_1+(f_2+(f_3+(f_4+\cdots \,\,\cdots)))$.
\end{enumerate}
The use of more general linearly ordered indexing sets also gives a clear way to define sums of the form $\sum_{k>j}f_{j,k}$ for natural numbers $j$ and $k$. Formally, this sum is indexed by the suborder $\{(j,k)\in\bbn^2\mid k>j\}$ of $\bbn^2$ with the lexicographical ordering.
\end{remark}

\begin{lemma}\label{negationlemma}
If $J$ is a countable linearly ordered set and $\{f_j\}_{j\in J}\in\loopjn(X,x_0)$, then $-\sum_{j\in J}f_j\simeq_{ir} \sum_{j\in J}-f_j$.
\end{lemma}

\begin{proof}
Applying Theorem \ref{cubeshuffletheorem} to both sums in the statement of the lemma, we may assume that $J=\bbn$. Theorem \ref{cubeshuffletheorem} also implies that $\sum_{j\in \bbn}f_j+\sum_{j\in \bbn}-f_j$ is homotopic to $\sum_{j\in J}(f_j+(-f_j))$ in $Y=\bigcup_{j\in\bbn}\im(f_j)$. Since $f_j+(-f_j)\simeqir c$ where $c$ is the constant map, we fix such null-homotopies and take their infinite horizontal concatenation. Hence, $\sum_{j\in J}(f_j+(-f_j))$ is null-homotopic in $Y$. It follows that $\sum_{j\in \bbn}f_j+\sum_{j\in \bbn}-f_j$ is null-homotopic in $Y$. Also, $-\sum_{j\in \bbn}f_j+\sum_{j\in \bbn}f_j$ is null-homotopic in $Y$. Hence, $-\sum_{j\in \bbn}f_j$ is image-relative homotopic to the first map in the following chain of image-relative homotopies.
\[-\sum_{j\in \bbn}f_j+\left(\sum_{j\in \bbn}f_j+\sum_{j\in \bbn}-f_j\right)\simeqir\left(-\sum_{j\in \bbn}f_j+\sum_{j\in \bbn}f_j\right)+\sum_{j\in \bbn}-f_j\simeqir \sum_{j\in \bbn}-f_j\]
Since all maps in this chain have the same image, the weak notion of transitivity from Remark \ref{weaktransitiveremark} applies to give $-\sum_{j\in \bbn}f_j\simeq_{ir} \sum_{j\in \bbn}-f_j$.
\end{proof}

Doubly-indexed infinite sums also fall within the scope of Theorem \ref{cubeshuffletheorem}.

\begin{lemma}\label{doublesumlemma}
Suppose $n\geq 2$, $J,K$ are countable linearly ordered sets, and $J\times K$ is given the lexicographical ordering. If $\{f_{j,k}\}_{(j,k)\in J\times K}$ is a doubly-indexed collection in $\loopn_{J\times K}(X,x_0)$, then 
\[\sum_{j\in J}\left(\sum_{k\in K}f_{j,k}\right)\simeq_{ir}\sum_{k\in K}\left(\sum_{j\in J}f_{j,k}\right).\]
\end{lemma}

A relevant special case is the following linearity condition, which follows from Lemmas \ref{negationlemma} and \ref{doublesumlemma}.

\begin{corollary}\label{linearityofsumscor}
Let $J$ be a countable linearly ordered set and $\{f_j\}_{j\in J}$ and $\{g_j\}_{j\in J}$ be elements of $\loopjn(X,x_0)$. If  $k_j,m_j$ are integers for all $j\in J$, then \[\sum_{j\in J}k_jf_j+m_jg_j\simeq_{ir}k_j\sum_{j\in J}f_j+m_j\sum_{j\in J}g_j.\]
\end{corollary}

\section{Combining Infinite Sums and Whitehead Products}\label{sectionidentity}

We wish to apply the standard relations in Remark \ref{relations} infinitely many times within a single step of algebraic reasoning. To do so, we show that each one is realized by an image-relative homotopy. Certainly, this could be achieved by analyzing the original proofs of these relations, c.f. \cite{whiteheadjhc,whiteheadGWproducts}. However, since we are already assuming the standard relations hold (particularly, in wedges of spheres), we can avoid restating the classical proofs by analyzing the formula for the Whitehead bracket on maps. Throughout this subsection, $X$ will denote a topological space with basepoint $x_0$ and we assume that $p,q,r\geq 2$.

\begin{proposition}\label{welldefinedprop}
Let $f_1,f_2\in\Omega^p(X)$, $g_1,g_2\in\Omega^q(X)$, $H:S^p\times\ui\to X$ be a based homotopy from $f_1$ to $f_2$, and $G:S^q\times\ui\to X$ be a based homotopy from $g_1$ to $g_2$. Then $[ f_1,g_1] \simeq [ f_2,g_2] $ by a based homotopy with image equal to $\im(H)\cup \im(G)$. In particular, if $f_1\simeqir f_2$ and $g_1\simeqir g_2$, then $[ f_1,g_1] \simeqir [ f_2,g_2] $.
\end{proposition}

\begin{proof}
There is a based homotopy $J:(S^p\vee S^q)\times\ui\to X$ from $f_1\vee g_1$ to $f_2\vee g_2$ with image equal to $\im(H)\cup \im (G)$. Let $K=J\circ (a_{p,q}\times id_{\ui})$ where $a_{p,q}:S^{p+q-1}\to S^p\vee S^q$ is the attaching map defined previously and note that $K$ is a based homotopy from $[ f_1,g_1]$ to $[ f_2,g_2]$. Since $a_{p,q}\times id_{\ui}$ is onto, we have $\im(K)=\im(H)\cup \im (G)$. The second statement follows immediately from the first.
\end{proof}

Since we may glue the terms of a sequential homotopy into a single homotopy via horizontal concatenation, we have the following consequence.

\begin{corollary}\label{irwhiteheadcor}
Suppose $\{H_j\}_{j\in\bbn}$ is a sequential homotopy from $\{f_j\}_{j\in\bbn}$ to $\{f_{j}'\}_{j\in\bbn}$, where $\{f_j\}_{j\in\bbn},\{f_{j}'\}_{j\in\bbn}$ are elements of $\Omega^{p}_{\bbn}(X,x_0)$ and $\{G_j\}_{j\in\bbn}$ is a sequential homotopy from $\{g_j\}_{j\in\bbn}$ to $\{g_{j}'\}_{j\in\bbn}$ where $\{g_j\}_{j\in\bbn},\{g_{j}'\}_{j\in\bbn}$ are elements of $\Omega^{q}_{\bbn}(X,x_0)$. Then $\{[ f_j,g_j]\}_{j\in\bbn}$ and $\{[ f_{j}',g_{j}']\}_{j\in\bbn}$ are elements of $\Omega^{p+q-1}_{\bbn}(X,x_0)$ and we have
\[\sum_{j=1}^{\infty}[ f_j,g_j] \simeq \sum_{j=1}^{\infty}[ f_j',g_j']\] by a homotopy in $\bigcup_{j\in\bbn}\im(H_j)\cup \im(G_j)$. In particular, if $f_j\simeqir f_{j'}$ and $g_j\simeq_{ir} g_{j}'$ for all $j\in\bbn$, then $\ds\sum_{j=1}^{\infty}[ f_j,g_j] \simeqir \sum_{j=1}^{\infty}[ f_j',g_j']$.
\end{corollary}

\begin{proposition}[Bilinearity]\label{bilinearlityprop}
Let $f_1,f_2\in \Omega^p(X)$ and $g_1,g_2\in \Omega^q(X)$. Then 
\begin{enumerate}
\item $[ f_1+ f_2,g_1]\simeq_{ir} [ f_1,g_1]+ [ f_2,g_1]$,
\item $[ f_1,g_1+ g_2]\simeq_{ir} [ f_1,g_1]+ [ f_1,g_2]$.
\end{enumerate}
\end{proposition}

\begin{proof}
We prove the first statement. The second is proved by a symmetric argument. Let $Y=S^p\vee S^p\vee S^q$ and let $\gamma_1,\gamma_2:S^p\to Y$ and $\delta:S^q\to Y$ be the inclusions of the wedge-summands. The standard bilinearity relation applied within $Y$ gives a map $H:S^{p+q-1}\times\ui\to Y$ that serves as a based homotopy $[ \gamma_1+ \gamma_2,\delta]\simeq [\gamma_1,\delta]+[ \gamma_2,\delta]$. Here, $[\gamma_1+ \gamma_2,\delta]=(\nabla_{p}\vee id_{S^q})\circ a_{p,q}$ where $\nabla_p:S^p\to S^p\vee S^p$ is the map used to define concatenation of $p$-loops. Also, \[[ \gamma_1,\delta]+ [ \gamma_2,\delta]=((\gamma_1\vee \delta)\vee (\gamma_2\vee \delta))\circ (a_{p,q}\vee a_{p,q})\circ \nabla_{p+q-1}.\]
Define $F=f_1\vee f_2\vee g_1$ as a map $Y\to X$. The definition of Whitehead product gives \[F\circ [ \gamma_1+ \gamma_1,\delta]= (((f_1\vee f_2)\circ\nabla_{p})\vee g_1)\circ a_{p,q}=[ f_1+ f_2,g_1].\]
Since $F\circ (\gamma_i\vee \delta)\circ a_{p,q}=[ f_i,g_1]$, we have $F\circ ([ \gamma_1,\delta]+ [ \gamma_2,\delta]) = [ f_1,g_1]+ [f_2,g_1]$. Therefore, $F\circ H$ is the desired homotopy. Moreover, since $H$ is surjective, we have $\im (F\circ H)=\im(f_1)\cup\im(f_2)\cup \im(g_1)$, 
\end{proof}

\begin{lemma}\label{bilinearitylemma1}
Let $\{f_j\}_{j\in\bbn}$, $\{f_j'\}_{j\in\bbn}$ be elements of $\Omega^{p}_{\bbn}(X,x_0)$ and $\{g_j\}_{j\in\bbn}$, $\{g_j'\}_{j\in\bbn}$ be elements of $\Omega^{q}_{\bbn}(X,x_0)$. Then
\begin{enumerate}
\item $\ds \sum_{j=1}^{\infty}[ f_j+f_{j}',g_j] \simeq_{ir}\sum_{j=1}^{\infty}[ f_j,g_j] +\sum_{j=1}^{\infty}[ f_{j}',g_j] $,
\item $\ds \sum_{j=1}^{\infty}[ f_j,g_j+g_{j}'] \simeq_{ir}\sum_{j=1}^{\infty}[ f_j,g_j] +\sum_{j=1}^{\infty}[ f_{j},g_j']$.
\end{enumerate}
\end{lemma}

\begin{proof}
By Proposition \ref{bilinearlityprop}, for each $j\in\bbn$, we have an image-relative homotopy $[ f_j+f_{j}',g_j] \simeqir [ f_j,g_j]+[f_{j}',g_j]$. The infinite horizontal concatenation of these homotopies gives image-relative homotopy $\ds \sum_{j=1}^{\infty}[ f_j+f_{j}',g_j]\simeq_{ir} \sum_{j=1}^{\infty}[ f_j,g_j] +[ f_{j}',g_j]$. Vertically composing this with an infinite commutativity homotopy $\sum_{j=1}^{\infty}[ f_j,g_j]+[ f_{j}',g_j] \simeq_{ir}\sum_{j=1}^{\infty}[ f_j,g_j] +\sum_{j=1}^{\infty}[ f_{j}',g_j] $ from Corollary \ref{linearityofsumscor} constructs the desired image-relative homotopy (recall the weak transitivity of $\simeqir$ from Remark \ref{weaktransitiveremark}). A symmetric argument gives the homotopy in (2).
\end{proof}

\begin{lemma}\label{whnegationlemma}
For all $f\in\Omega^p(X)$ and $g\in \Omega^q(X)$, any two of the following three maps are image-relative homotopic to each other: $[-f,g]$, $-[ f,g]$, and $[ f,-g]$.
\end{lemma}

\begin{proof}
We show that $[-f,g]\simeqir-[ f,g]$ holds. The second relation $[f,-g]\simeqir-[ f,g]$ follows from a symmetric argument and then the third relation $[-f,g]\simeqir[ f,-g]$ holds by the weak transitivity of $\simeqir$ described in Remark \ref{weaktransitiveremark}. 

By Proposition \ref{bilinearlityprop}, we have $[ f,g]+[ -f,g]\simeqir [ f+ (-f),g]$ and if $c$ denotes the constant map, then $c\simeqir f+ (-f)$. Thus by Proposition \ref{welldefinedprop}, we have $[f+(-f),g]\simeqir [ c,g]$. We check that $[ c,g]$ is null-homotopic in $\im(g)$. Let $\gamma:S^q\to S^p\vee S^q$ and $\delta:S^q\to S^p\vee S^q$ be the inclusion maps and $c':S^p\to S^p\vee S^q$ be the constant map. The usual identity relation from Remark \ref{relations} gives that $ [ c',\delta]=(c'\vee \delta)\circ [ \gamma,\delta]:S^{p+q-1}\to S^{p}\vee S^{q}$ is null-homotopic in $S^p\vee S^q$. Thus $[ c,g]=(c\vee g)\circ [ \gamma,\delta]=(id\vee g)\circ (c'\vee \delta)\circ[ \gamma,\delta]$ is null-homotopic in $\im(c)\cup\im(g)=\im(g)$. Vertical composition of the above homotopies shows that $[ f,g]+[ -f,g]\simeqir c$. Thus we have the following chain of image-relative homotopies: \[-[ f,g]\simeqir -[f,g]+([f,g]+[ -f,g])\simeqir(-[f,g]+[f,g])+[ -f,g]\simeqir[-f,g].\]Since all maps in this chain have the image equal to $\im(f)\cup\im(g)$, we conclude that $[-f,g]\simeqir-[ f,g]$.
\end{proof}

Combining the previous two lemmas, we have the following.

\begin{corollary}\label{linearitywhiteheadcor}
Let $J$ be a countable linearly ordered set, $\{f_j\}_{j\in J}$ be an element of $\Omega_{J}^{p}(X,x_0)$, and $\{g_j\}_{j\in J}$ be an element of $\Omega_{J}^{q}(X,x_0)$. If $\{m_j\}_{j\in J}$ is a collection of integers, then any two of the following three maps are image-relative homotopic to each other: $\ds\sum_{j\in J}m_j[f_j,g_j]$, $\ds\sum_{j \in J}[m_jf_j,g_j]$, and $\ds\sum_{j\in J}[f_j,m_jg_j]$.
\end{corollary}

\begin{remark}
One should not expect the following infinite version of bilinearity to hold: $[\sum_{j=1}^{\infty}f_j,g]\simeq\sum_{j=1}^{\infty}[f_j,g]$. Indeed, for the sum $\sum_{j=1}^{\infty}[f_j,g]$ to be a well-defined map, $g$ must be the constant map (if $g$ is non-constant, the sequence $\{[f_j,g]\}_{j\in\bbn}$ will not converge to the basepoint). The next two lemmas provide the appropriate versions of an infinite application of bilinearity.
\end{remark}

\begin{lemma}[Infinite Bilinearity]\label{bilinearitylemma2}
If $\{f_j\}_{j\in\bbn}$ is an element of $\Omega_{\bbn}^{p}(X,x_0)$ and $\{g_j\}_{j\in\bbn}$ is an element of $\Omega_{\bbn}^{q}(X,x_0)$, then \[\left[\sum_{j=1}^{\infty}f_j,\sum_{j=1}^{\infty}g_j \right]\simeq_{ir}\sum_{j=1}^{\infty}\left[ f_j,g_j\right] +\sum_{j=1}^{\infty}\left[ f_j, \sum_{k>j}g_k\right] +\sum_{j=1}^{\infty}\left[ \sum_{k>j}f_k,g_j\right] .\]
\end{lemma}

\begin{proof}
Let $F_{\geq j}=\sum_{k\geq j}f_k$ and $G_{\geq j}=\sum_{k\geq j}g_k$ and note that $F_{\geq 1}=\sum_{j=1}^{\infty}f_j$ and $G_{\geq 1}=\sum_{j=1}^{\infty}g_j$. Fix $j\in\bbn$. Applying Proposition \ref{welldefinedprop} and Proposition \ref{bilinearlityprop} multiple times, we obtain the following chain of image-relative homotopies which all may be chosen to have image in $Y_j=\im(F_{\geq j})\cup \im(G_{\geq j})$. 
\begin{eqnarray*}
[ F_{\geq j},G_{\geq j}]  & \simeqir & [ f_j+ F_{\geq j+1}, G_{\geq j}] \\
 &\simeqir & [ f_j, G_{\geq j}]+[ F_{\geq j+1},G_{\geq j}]\\
 &\simeqir & [ f_j, g_j+ G_{\geq j+1}]+[ F_{\geq j+1},g_j+ G_{\geq j+1}]\\ 
 &\simeqir & [ f_j,g_j]+[ f_j, G_{\geq j+1}]+[ F_{\geq j+1},g_j] +[ F_{\geq j+1},G_{\geq j+1}]
\end{eqnarray*}
For each $j\in\bbn$, let $L_{j}$ denote the resulting image-relative homotopy from $[ F_{\geq j},G_{\geq j}]$ to $([ f_j,g_j]+[ f_j, G_{\geq j+1}]+[ F_{\geq j+1},g_j]) +[ F_{\geq j+1},G_{\geq j+1}]$. Since the sequences $\{F_{\geq j}\}_{j\in\bbn}$ and $\{G_{\geq j}\}_{j\in\bbn}$ converge to $x_0$, so does $\{L_j\}_{j\in\bbn}$. The infinite vertical concatenation (recall Remark \ref{infverticalremark}) of the sequence $\{L_j\}_{j\in\bbn}$ gives
\[[ F_{\geq 1},G_{\geq 1}]\simeqir \sum_{j=1}^{\infty}[ f_j,g_j]+[ f_j, G_{\geq j+1}] +[ F_{\geq j+1},g_j].\]
Finally, we apply an infinite commutativity homotopy from Corollary \ref{linearityofsumscor} to see that \[\sum_{j=1}^{\infty}[ f_j,g_j]+[ f_j, G_{\geq j+1}] +[ F_{\geq j+1},g_j]\simeqir \sum_{j=1}^{\infty}[ f_j,g_j] +\sum_{j=1}^{\infty}[f_j, G_{\geq j+1}] +\sum_{j=1}^{\infty}[ F_{\geq j+1},g_j].\]
Since these homotopies all have image in $Y_1$ their vertical composition gives
\[[ F_{\geq 1},G_{\geq 1}]\simeqir \sum_{j=1}^{\infty}[ f_j,g_j] +\sum_{j=1}^{\infty}[f_j, G_{\geq j+1}] +\sum_{j=1}^{\infty}[ F_{\geq j+1},g_j].\]
\end{proof}

If we perform an infinite sum (indexed over $i\in \bbn$) of the image-relative homotopies from Lemma \ref{bilinearitylemma2} and apply linearity of infinite sums (Corollary \ref{linearityofsumscor}), then we obtain the following.

\begin{corollary}\label{bilinearitycor1}
If $\{f_{i,j}\}_{(i,j)\in\bbn^2}$ is an element of $\Omega^{p}_{\bbn^2}(X,x_0)$ and $\{g_{i,j}\}_{(i,j)\in\bbn^2}$ is an element of $\Omega^{q}_{\bbn^2}(X,x_0)$, then $\ds\sum_{i\in\bbn}\left[\sum_{j=1}^{\infty}f_{i,j},\sum_{j=1}^{\infty}g_{i,j}\right]$ is image-relative homotopic to \[ \sum_{(i,j)\in\bbn^2}[f_{i,j},g_{i,j}]+
\sum_{(i,j)\in\bbn^2}\left[f_{i,j},\sum_{k>j}g_{i,k}\right]+\sum_{(i,j)\in\bbn^2}\left[\sum_{k>j}f_{i,k},g_{i,j}\right].\]
\end{corollary}

\begin{proposition}[Graded symmetry]\label{gradedsymprop}
For all $f\in \Omega^p(X)$ and $g\in \Omega^q(X)$, we have $[ f,g]\simeqir (-1)^{pq}[ g,f]$.
\end{proposition}

\begin{proof}
Let $\tau:S^p\times S^q\to S^q\times S^p$ be the twist homeomorphism, which restricts to a homeomorphism $S^p\vee S^q\to S^q\vee S^p$. Let $\gamma:S^p\to S^p\vee S^q$ and $\delta:S^q\to S^p\vee S^q$ be the inclusion maps. Then $\tau\circ[ \delta,\gamma]:S^{p+q-1}\to S^p\vee S^q$ is the map satisfying $[ g,f]=(f\circ g) \circ \tau\circ[ \delta,\gamma]$. By the classical theory, we have $[\gamma,\delta]\simeq (-1)^{pq}(\tau\circ[ \delta,\gamma])$ by a homotopy $K:S^{p+q-1}\times \ui\to S^p\vee S^q$. Now $(f\vee g)\circ K$ is a homotopy from $[ f,g]$ to $(-1)^{pq}[ g,f]$ with image in $\im(f)\cup \im(g)$, proving the proposition.
\end{proof}

Applying a horizontal concatenation of graded symmetry homotopies from Proposition \ref{gradedsymprop}, we obtain the following.

\begin{corollary}\label{gradedsymcor}
Let $J$ be a countable linearly ordered set, $\{f_{j}\}_{j\in J}$ be an element of $\Omega_{J}^{p}(X,x_0)$, and $\{g_{j}\}_{j\in J}$ be an element of $\Omega_{J}^{q}(X,x_0)$. Then any two of the following three maps are image-relative homotopic to each other: $\sum_{j\in J}[f_j,g_j]$, $\sum_{j \in J}(-1)^{pq}[g_j,f_j]$, and $(-1)^{pq}\sum_{j \in J}[g_j,f_j]$.
\end{corollary}

The identity in the next theorem is a culmination of the previous identities in the case $p=q=n$ and will be used in the next section. Our use of image-relative homotopies makes the proof a straightforward combination of Lemma \ref{negationlemma}, Corollary \ref{bilinearitycor1}, and Corollary \ref{gradedsymcor}.

\begin{theorem}\label{mainidentitylemma}
If $\{f_{i,j}\}_{(i,j)\in\bbn^2}$ and $\{g_{i,j}\}_{(i,j)\in\bbn^2}$ are elements of $\Omega^{n}_{\bbn^2}(X)$, then $\ds\sum_{i=1}^{\infty}\left[\sum_{j=1}^{\infty}f_{i,j},\sum_{j=1}^{\infty}g_{i,j}\right]$ is an element of $\Omega^{2n-1}(X)$ that is image-relative homotopic to \[ \sum_{(i,j)\in\bbn^2}[f_{i,j},g_{i,j}]+
\sum_{(i,j)\in\bbn^2}\left[f_{i,j},\sum_{k>j}g_{i,k}\right]+(-1)^{n^2}\sum_{(i,j)\in\bbn^2}\left[g_{i,j},\sum_{k>j}f_{i,k}\right].\]
\end{theorem}

Although our application in the next section does not require the graded Jacobi identity, it is certainly important for future work in this area. Therefore, we complete this section by noting how it relates to infinite sums.
\begin{proposition}[Jacobi Identity]
For all $f\in\Omega^p(X)$, $g\in \Omega^q(X)$, $h\in \Omega^r(X)$, the map 
\[(-1)^{pr}[[ f,g],h]+ (-1)^{pq}[[ g,h] ,f]+(-1)^{rq}[[ h,f],g]\in \Omega^{p+q+r-2}(X)\] is null-homotopic in by an image-relative homotopy.
\end{proposition}

\begin{proof}
Let $Y=S^p\vee S^q\vee S^r$ and $\alpha:S^p\to Y$, $\beta:S^q\to Y$, and $\gamma:S^r\to Y$ be the inclusions of the wedge-summands. Then \[(-1)^{pr}[[ \alpha,\beta],\gamma]+ (-1)^{pq}[[ \beta,\gamma] ,\alpha]+ (-1)^{rq}[[ \gamma,\alpha],\beta]\]is null-homotopic in $Y$ by a hull-homotopy $K:S^{p+q+r-2}\times\ui\to Y$. Now $(f\vee g\vee h)\circ K$ is the desired null-homotopy.
\end{proof}

\begin{corollary}
Let $\{f_j\}_{j\in\bbn}$ be an element of $\Omega_{\bbn}^{p}(X,x_0)$, $\{g_j\}_{j\in\bbn}$ be an element of $\Omega_{\bbn}^{q}(X,x_0)$, $\{h_j\}_{j\in\bbn}$ be an element of $\Omega_{\bbn}^{r}(X,x_0)$. Then \[(-1)^{pr}\sum_{j\in\bbn}[[ f_j,g_j],h_j]+ (-1)^{pq}\sum_{j\in\bbn}[[ g_j,h_j] ,f_j]+(-1)^{rq}\sum_{j\in\bbn}[[ h_j,f_j],g_j]\] is null-homotopic by an image-relative homotopy.
\end{corollary}

\section{Infinite sum closures and Whitehead products in Shrinking Wedges}

\subsection{Shrinking wedges of spaces}

\begin{definition}\label{swdef}
For a countably infinite collection $\{(X_j,x_j)\}_{j\in J}$ of based spaces, let $\sw_{j\in J}X_j$ denote the wedge sum $\bigvee_{j\in J}X_j$, i.e. one-point union, where the points $x_j$ are identified to a canonical basepoint $x_0$. We give $\sw_{j\in J}X_j$ the topology consisting of sets $U$ such that $U\cap X_j$ is open in $X_j$ for all $j\in J$ and if $x_0\in U$, then $X_j\subseteq U$ for all but finitely many $j\in J$. We call $\sw_{j\in J}X_j$ the \textit{shrinking wedge} of the collection $\{(X_j,x_j)\}_{j\in J}$.
\end{definition}

\begin{definition}
The \textit{$n$-dimensional earring space} $\bbe_n$ is the shrinking wedge of $n$-spheres, that is $\bbe_n=\sw_{j\in\bbn}X_j$ where $X_j=S^n$ for all $j\in\bbn$.
\end{definition}

When a shrinking wedge $\sw_{j\in \bbn}X_j$ is indexed by the natural numbers, let $X_{\leq k}=\bigvee_{j=1}^{k}X_j$ be the finite wedge and let $X_{\geq k}=\sw_{j=k}^{\infty}X_j$ and $X_{>k}=\sw_{j=k+1}^{\infty}X_j$ denote cofinal shrinking-wedge subspaces. For any non-empty subset $F\subseteq \bbn$, the set $\bigcup_{j\in F}X_j$ is canonically a retract of $\sw_{j\in\bbn}X_j$ and we identify $\pi_n(\bigcup_{k\in F}X_k)$ with the corresponding subgroup of $\pi_n(\sw_{j\in\bbn}X_j)$.

If $p_j:\sw_{j\in\bbn}X_j\to X_j$ denotes the canonical retraction for each $j\in\bbn$, then there is an embedding $\sigma:\sw_{j\in\bbn}X_j\to \prod_{j\in\bbn}X_j$ into the infinite direct product given by $\sigma(x)=(p_j(x))_{j\in\bbn}$. We sometimes identify $\sw_{j\in\bbn}X_j$ as a subspace of $\prod_{j\in\bbn}X_j$ so that $\sigma$ is an inclusion map. Under this identification, the long exact sequence of the pair $\left(\prod_{j\in\bbn}X_j,\sw_{j\in\bbn}X_j\right)$ breaks into short exact sequences
\[\xymatrix{0 \ar[r] & \pi_{n+1}\left(\prod_{j\in\bbn}X_j,\sw_{j\in\bbn} X_j\right) \ar[r]^-{\partial} & \pi_n(\sw_{j\in\bbn}X_j) \ar[r]^-{\sigma_{\#}} & \pi_n(\prod_{j\in\bbn}X_j) \ar[r] & 0
}\]
when $n\geq 2$. Moreover, there is a canonical splitting homomorphism $\zeta:\pi_n(\prod_{j\in\bbn}X_j)\to \pi_n(\sw_{j\in\bbn}X_j)$ given by $\zeta([(f_j)_{j\in\bbn}])=\left[\sum_{j\in\bbn}f_j\right]$ where $(f_j)_{j\in \bbn}:S^n\to \prod_{j\in\bbn}f_j$ represents the map satisfying $p_k\circ(f_j)_{j\in\bbn}=f_k$ for all $k\in\bbn$ (see \cite{Kawamurasuspensions}). Hence, \[\ker(\sigma_{\#})\cong\pi_{n+1}\left(\prod_{j\in\bbn}X_j,\sw_{j\in\bbn} X_j\right)\] and we have a natural splitting of $\pi_n(\sw_{j\in\bbn}X_j)$:
\begin{eqnarray*}
\pi_n\left(\sw_{j\in\bbn}X_j\right) &\cong& \pi_{n+1}\left(\prod_{j\in\bbn}X_j,\sw_{j\in\bbn}X_j\right)\oplus \pi_n\left(\prod_{j\in\bbn}X_j\right)\\
&\cong& \pi_{n+1}\left(\prod_{j\in\bbn}X_j,\sw_{j\in\bbn}X_j\right)\oplus \prod_{j\in\bbn}\pi_n(X_j)
\end{eqnarray*}

The following theorem is a special case of the main result of \cite{EK00higher}. An alternative proof is given in \cite{BrazSequentialnconn}.

\begin{theorem}[Eda-Kawamura]\label{connectedtheorem}
If $n\geq 2$ and $X=\sw_{j\in\bbn}X_j$ is a shrinking wedge of $(n-1)$-connected CW-complexes, then $X$ is $(n-1)$-connected and $\sigma_{\#}:\pi_n(X)\to \pi_n(\prod_{j\in\bbn}X_j)$ is an isomorphism. Thus $\pi_n(X)\cong\prod_{j\in\bbn}\pi_n(X_j)$.
\end{theorem}

Theorem \ref{connectedtheorem} implies that if each $X_j$ is $(n-1)$-connected and $\alpha\in\pi_n\left(\sw_{j\in\bbn}X_j\right)$, then there are unique elements $\alpha_j\in \pi_n(X_j)$ such that $\alpha=\left[\sum_{j=1}^{\infty}f_j\right]$ for any choice of representing maps $f_j\in \alpha_j$. In the case of the $n$-dimensional earring, we have $\pi_{n}(\bbe_n)\cong \pi_n\left(\prod_{j\in\bbn}S^n\right)\cong \bbz^{\bbn}$. Moreover, if $\ell_j:S^n\to \bbe_n$ denotes the inclusion of the $j$-th sphere, then every element of $\pi_n(\bbe_n)$ may be represented uniquely as $\left[\sum_{j=1}^{\infty}m_j\ell_j\right]$ for a sequence of integers $\{m_{j}\}_{j\in\bbn}$.

\begin{example}\label{archipelagoexample}
Let $\bbe_n$ be the $n$-dimensional earring space for some $n\geq 2$ and let $\ell_j:S^n\to \bbe_n$ denote the inclusion of the $j$-th sphere. Take $\bba=\bbe_n\cup_{\ell_j} \bigcup_{j\in\bbn}e_{j}^{n+1}$ to be the space obtained by attaching, for each $j\in\bbn$, an $(n+1)$-cell $e_{j}^{n+1}$ to $\bbe_n$ with the attaching map $\ell_j$. Then $\ell_j$ is null-homotopic in $\bba$ for all $j\in\bbn$. However, $f=\sum_{j=1}^{\infty}\ell_j$ is not null-homotopic in $\bba$. Indeed, every map $H:S^n\times \ui\to \bba$ has compact image and therefore lies in one of the subspaces $\bba_k=\bbe_n\cup \bigcup_{j=1}^{k}e_{j}^{n+1}$ but $f$ is not null-homotopic in any $\bba_k$ since $(\bbe_n)_{>k}$ is a retract of $\bba_k$ and the image $\sum_{j=k+1}^{\infty}\ell_j$ of $f$ under this retraction is not null-homotopic in $(\bbe_n)_{>k}$. Thus, in $\pi_n(\bba)$, the infinite sum operation on homotopy classes is not well defined. While this possibility may appear startling at first glance it does not hinder our ability to describe the group since Theorem \ref{connectedtheorem} makes the computation $\pi_n(\bba)\cong \prod_{j\in\bbn}\bbz/\bigoplus_{j\in\bbn}\bbz$ straightforward.
\end{example}

\subsection{Infinite sum closures}

\begin{definition}\label{infiniteclosuredef1}
Suppose $(X,x_0)$ is a based space and $n\geq 2$.
\begin{enumerate}
\item We say that a subset $S\subseteq \pi_n(X,x_0)$ is \textit{closed under infinite-sums} if for all $\{f_j\}_{j\in\bbn}\in\loopn_{\bbn}(X,x_0)$ such that $[f_j]\in S$ for all $j\in\bbn$, we have $\left[\sum_{j=1}^{\infty}f_j\right]\in S$. 
\item The \textit{infinite-sum closure of a subset} $S\subseteq \pi_n(X,x_0)$ is the smallest subgroup $cl_{\Sigma}(S)$ of $\pi_n(X,x_0)$ that contains $S$ and which is closed under infinite sums, that is,
\[cl_{\Sigma}(S)=\bigcap\{H\leq \pi_n(X,x_0)\mid S\subseteq H\text{ and }H\text{ is closed under infinite sums}\}.\]
\item Given a subset $S\subseteq \pi_n(X,x_0)$ and the subgroup $\lb S\rb$ of $\pi_n(X,x_0)$ generated by $S$, we write $\llbracket S\rrbracket$ to denote the subset of $\pi_n(X,x_0)$ consisting of infinite sums $\left[\sum_{j=1}^{\infty}f_j\right]$ where $\{f_j\}_{j\in\bbn}\in\loopn_{\bbn}(X,x_0)$ and $[f_j]\in \lb S\rb $ for all $j\in\bbn$.
\end{enumerate}
\end{definition}

\begin{remark}
It is important to note that ``closure under infinite sums" refers to representing maps and not infinite sums of homotopy classes (which we have not defined). Recall that $\loopn_{\bbn}(X,x_0)$ is the set of sequences $\{f_j\}_{j\in\bbn}$ that converge to $x_0$. Hence, a set $S\subseteq \pi_n(X,x_0)$ is closed under infinite-sums if and only if it is closed under homotopy classes of well-defined infinite sum maps $\sum_{j=1}^{\infty}f_j$ where the summands $f_j$ represent elements of $S$. If a sequence $\{f_j\}_{j\in\bbn}$ of maps representing elements of $S$ does not converge to $x_0$, then it is not relevant to whether or not $S$ is closed under infinite sums. However, even if $\{f_j\}_{j\in\bbn}$ does not converge to $x_0$, it is possible that $\{g_j\}_{j\in\bbn}$ is another sequence which \textit{does} converge to $x_0$ and for which $[f_j]=[g_j]\in S$ for all $j\in\bbn$. In this case, we must have $\left[\sum_{j\in\bbn}g_j\right]\in S$ if $S$ is closed under infinite sums. Recall that is also possible that two well-defined infinite sum maps $\sum_{j=1}^{\infty}f_j$ and $\sum_{j=1}^{\infty}g_j$ have distinct homotopy classes even if $[f_j]=[g_j]\in S$ for all $j$. In this case, for $S$ to be closed under infinite-sums, it must be that $S$ contains both of the homotopy classes $\left[\sum_{j=1}^{\infty}f_j\right]$ and $\left[\sum_{j=1}^{\infty}g_j\right]$.
\end{remark}

Certainly, $\pi_n(X,x_0)$ is closed under infinite sums and the property of being ``closed under infinite sums" is closed under forming arbitrary intersections of subgroups of $\pi_n(X,x_0)$. Thus, for any given subset $S$ of $\pi_n(X,x_0)$, $cl_{\Sigma}(S)$ is a well-defined subgroup of $\pi_n(X,x_0)$.

\begin{proposition}\label{inclusionprop}
For any non-empty subset $S\subseteq \pi_n(X,x_0)$, $\llbracket S\rrbracket$ is a subgroup of $\pi_n(X,x_0)$ and we have subgroup inclusions $\langle S\rangle\leq \llbracket S\rrbracket\leq cl_{\Sigma}(S)\leq \pi_n(X,x_0)$.
\end{proposition}

\begin{proof}
Sequences $\{f_j\}_{j\in\bbn}$ for which $[f_j]\in \langle S\rangle$ and for which all but finitely many $f_j$ are constant maps show that $\langle S\rangle\subseteq \llbracket S\rrbracket$. Also, $cl_{\Sigma}(S)$ is a subgroup of $\pi_n(X,x_0)$ by definition and we have $\llbracket S\rrbracket\subseteq cl_{\Sigma}(S)$ by the definition of $\llbracket S\rrbracket$. Lastly, we check that $\llbracket S\rrbracket$ is a subgroup of $\pi_n(X,x_0)$. If $\{f_j\}_{j\in\bbn}$ and $\{g_j\}_{j\in\bbn}$ are elements of $\loopn_{\bbn}(X,x_0)$ and $[f_j],[g_j]\in \lb S\rb $ for all $j\in\bbn$, then the fact that $[f_j+(-1)g_j]\in \lb S\rb$ for all $j\in\bbn$ and that $\sum_{j\in\bbn}f_j+(-1)\sum_{j\in\bbn}g_j\simeq\sum_{j\in\bbn}(f_j+(-1)g_j)$ shows that the difference $\left[\sum_{j=1}^{\infty}f_j\right]-\left[\sum_{j=1}^{\infty}g_j\right]$ lies in $\llbracket S\rrbracket$. We conclude that $\llbracket S\rrbracket$ is a subgroup of $\pi_n(X,x_0)$.
\end{proof}

It is not obviously true that the subgroups $\llbracket S\rrbracket$ and $cl_{\Sigma}(S)$ are equal. In fact, this appears to be false in some scenarios. Rather, one should consider the operation $H\mapsto \llbracket H\rrbracket$ on subgroups $H\leq \pi_n(X,x_0)$ to be more akin to the sequential closure operator on sets in a topological space, which might need to be applied recursively until one arrives at the topological closure of the set. That being said, there are many natural situations in which $\llbracket S\rrbracket$ is closed under infinite sums and is therefore equal to $cl_{\Sigma}(S)$.

\begin{remark}
The constructions in Definition \ref{infiniteclosuredef1} are intended to be applied to spaces where infinite-sums can be formed only at a single ``wild" point. Indeed, the operation $S\mapsto cl_{\Sigma}(S)$ may not commute with path-conjugation when switching basepoints because a path-conjugated sequence $\{p\ast f_j\}_{j\in\bbn}$ of a sequence of maps $\{f_j\}_{j\in\bbn}\in \Omega^{n}_{\bbn}(X,x_0)$ does not converge to any point unless the path $p$ is constant. If one wishes to consider spaces in which non-trivial infinite sums can be formed at more than one point, then a modified definition of ``infinite-sum closure" might be needed.
\end{remark}

\begin{lemma}\label{sigmaclosedlemma}
Let $X=\sw_{j\in\bbn}X_j$ be a shrinking wedge of CW-complexes and $n\geq 2$. Then $\ker(\sigma_{\#}:\pi_n(X)\to \pi_n(\prod_{j\in\bbn}X_j))$ is closed under infinite sums.
\end{lemma}

\begin{proof}
Let $p_j:X\to X_j$ be the canonical retraction for each $j\in\bbn$ so that $\sigma(x)=(p_j(x))_{j\in\bbn}$. Let $\{f_k\}_{k\in\bbn}$ be an element of $\Omega^{n}_{\bbn}(X)$ such that $[f_k]\in \ker(\sigma_{\#})$ for all $k\in\bbn$. Thus $p_j\circ f_k$ is null-homotopic in $X_j$ for all pairs $(j,k)\in\bbn^2$. Since $X_j$ is a CW-complex, the sequence $\{p_j\circ f_k\}_{k\in\bbn}$ is sequentially null-homotopic in $X_j$. It follows that $p_j\circ \sum_{k=1}^{\infty}f_k=\sum_{k=1}^{\infty}p_j\circ f_k$ is null-homotopic in $X_j$ for each $j\in\bbn$. Since $\sigma$ trivializes the homotopy class of $\sum_{k=1}^{\infty}f_k$, we have $\left[\sum_{k=1}^{\infty}f_k\right]\in\ker (\sigma_{\#})$.
\end{proof}

\begin{proposition}\label{sumhomomoprop}
Let $X=\sw_{j\in\bbn}X_j$ be a shrinking wedge of based spaces and $n\geq 2$. The summation operation $\mathbf{S}_n:\prod_{k\in\bbn}\pi_{n}(X_{\geq k})\to \pi_n(X)$, $\mathbf{S}_n(([f_k])_{k\in\bbn})= \left[\sum_{k=1}^{\infty}f_k\right]$ is a well-defined homomorphism.
\end{proposition}

\begin{proof}
Suppose for each $k\in\bbn$, we have $[f_k]=[g_k]$ in $\pi_n(X_{\geq k})$ for maps $f_k,g_k\in \Omega^n(X_{\geq k})$. Then $\{f_k\}_{k\in\bbn}$ and $\{g_k\}_{k\in\bbn}$ converge to the basepoint and are sequentially homotopic. It follows from Proposition \ref{irsumsprop} that $\left[\sum_{k=1}^{\infty}f_k\right]=\left[\sum_{k=1}^{\infty}g_k\right]$. Thus $\mathbf{S}_n$ is well-defined. We have 
\begin{eqnarray*}
\mathbf{S}_n(([f_k])_{k\in\bbn}+([g_k])_{k\in\bbn}) &=& \mathbf{S}_n(([f_k+g_k])_{k\in\bbn})\\
&=&\left[\sum_{k=1}^{\infty}f_k+g_k\right]\\
&=& \left[\sum_{k=1}^{\infty}f_k\right]+\left[\sum_{k=1}^{\infty}g_k\right]\\
&=& \mathbf{S}_n(([f_k])_{k\in\bbn})+\mathbf{S}_n(([g_k])_{k\in\bbn})
\end{eqnarray*}
where the third equality follows from the linearity of infinite sums (Corollary \ref{linearityofsumscor}). Thus $\mathbf{S}_n$ is a homomorphism.
\end{proof}

\begin{definition}\label{wdef}
Suppose $X=\sw_{j\in\bbn}X_j$ is a shrinking wedge of based spaces, $n\geq 2$, and $f,g\in \Omega^n(X)$. We say the pair $(f,g)$ is \textit{wedge-disjoint} if there exist disjoint subsets $F,G$ of $\bbn$ such that $\im(f)\subseteq \bigcup_{j\in F}X_j$ and $\im(g)\subseteq \bigcup_{j\in G}X_j$. If $\alpha,\beta\in\pi_n(X)$, we say the pair $(\alpha,\beta)$ is \textit{wedge-disjoint} if there exist $f,g\in \Omega^n(X)$ such that $(f,g)$ is wedge-disjoint, $[f]=\alpha$, and $[g]=\beta$. Let \[\mcw_{2n-1}(X)=\left\llbracket\left\{[\alpha,\beta]\in \pi_{2n-1}(X)\mid (\alpha,\beta)\in\pi_{n}(X)^2\text{ is wedge-disjoint}\right\}\right\rrbracket.\]
\end{definition}

\begin{lemma}\label{replacementlemma}
Let $X=\sw_{j\in\bbn}X_j$ be a shrinking wedge of spaces where each $X_j$ has a countable neighborhood base of contractible sets at its basepoint. If $K$ is a compact based space and $\{f_k\}_{k\in\bbn}$ is a sequence of maps $f_k:K\to X$ that converges to the basepoint $x_0$ of $X$, then there exists a sequence $\{g_k\}_{k\in\bbn}$ of based maps $g_k:K\to X$ and a sequence of integers $\{m_k\}_{k\in\bbn}$ such that
\begin{enumerate}
\item $\{m_k\}_{k\in\bbn}\to\infty$,
\item $\im(g_k)\subseteq X_{\geq m_k}$ for all $k\in\bbn$,
\item $\{f_k\}_{k\in\bbn}$ and $\{g_k\}_{k\in\bbn}$ are sequentially homotopic.
\end{enumerate}
\end{lemma}

\begin{proof}
For each $j\in\bbn$, let $U_{j,2}\supseteq U_{j,3}\supseteq U_{j,4}\supseteq\cdots$ be a neighborhood base of contractible sets at the basepoint of $X_j$. Let $V_1=X$ and, for each $r\geq 2$, let $V_{r}=X_{\geq r}\vee \bigvee_{j=1}^{r-1}U_{j,r}$ so that $\{V_r\}_{r\geq 2}$ is a neighborhood base in $X$ at the wedgepoint. Note that the retraction $q_{r}:V_r\to X_{\geq r}$, which is the restriction of $p_{\geq k}$, is a based homotopy equivalence for each $r\geq 2$. Now, for each $k\in\bbn$, let $m_k$ be the smallest positive integer such that $\im(f_k)\subseteq V_{m_k}$ and note that $\{m_k\}_{k\in\bbn}\to\infty$ since $\{f_k\}_{k\in\bbn}$ converges to the basepoint. Let $g_k=q_{m_k}\circ f_k$ for each $k\in\bbn$. Then $\{g_k\}_{k\in\bbn}$ is the desired sequence since $f_k\simeq g_k$ in $V_{m_k}$ and $\im(g_k)\subseteq X_{\geq m_k}$.
\end{proof}

\begin{corollary}\label{replacementcorollary}
Lemma \ref{replacementlemma} also holds if each $X_j$ is a CW-complex.
\end{corollary}

\begin{proof}
Suppose $\{f_k\}_{k\in\bbn}$ is a sequence of based maps $K\to X$ that converges to the basepoint of a shrinking wedge $X=\sw_{j\in\bbn}X_j$ of CW-complexes $X_j$. Let $p_j:X\to X_j$ be the canonical retraction. For each $j\in\bbn$, the sequence $\{p_j\circ f_k\}_{k\in\bbn}$ of maps $K\to X_j$ converges to the basepoint of $X_j$. Therefore, $\bigcup_{k\in\bbn}\im(p_j\circ f_k)$ is compact and lies in a finite subcomplex $Y_j$ of $X_j$. Let $Y=\sw_{j\in\bbn}Y_j$. Now $\{f_k\}_{k\in\bbn}$ is a sequence of maps $K\to Y$ where $Y$ meets the hypotheses of Lemma \ref{replacementlemma} and the Corollary follows. 
\end{proof}

We will use Corollary \ref{replacementcorollary} several times to ``replace" an arbitrary sequence of maps $\{f_k\}_{k\in\bbn}\in \loopn_{\bbn}(X)$ with a sequentially homotopic sequence $\{g_k\}_{k\in\bbn}$ satisfying the conclusions stated in Lemma \ref{replacementlemma}.

\begin{lemma}\label{closurelemma}
Let $X=\sw_{j\in\bbn}X_j$ be a shrinking wedge of CW-complexes and $n\geq 2$. The subgroup $\mathcal{W}_{2n-1}(X)$ of $\pi_{2n-1}(X)$ is equal to the set of all elements of the form $\left[\sum_{i=1}^{\infty}[f_i,g_i]\right]$ where $\{f_i\}_{i\in\bbn},\{g_i\}_{i\in\bbn}\in\loopn_{\bbn}(X)$ and where each pair $(f_i,g_i)$ is wedge-disjoint.
\end{lemma}

\begin{proof}
Let $S=\{[\alpha,\beta]\in\pi_{2n-1}(X)\mid ( \alpha,\beta)\in \pi_{n}(X)^2\text{ is wedge-disjoint}\}$ so that $\mcw_{2n-1}(X)=\llbracket S\rrbracket$. Then a generic element of $\mcw_{2n-1}(X)$ has the form $\left[\sum_{i=1}^{\infty}w_i\right]$ where $\{w_i\}_{i\in\bbn}\in \Omega_{\bbn}^{n}(X)$ and $[w_i]\in \langle S\rangle$. By making the replacement from Corollary \ref{replacementcorollary} (without changing sequential homotopy class of $\{w_i\}_{i\in\bbn}$), we may assume there exists a sequence of integers $\{k_{i}\}_{i\in\bbn}\to\infty$ such that $\im(w_i)\subseteq X_{\geq k_{i}}$ for each $i\in\bbn$. Since $[w_i]\in \langle S\rangle$, we have $w_i\simeq \sum_{l=1}^{m_i}[f_{i,l},g_{i,l}]$ for wedge-disjoint pairs $(f_{i,l},g_{i,l})\in \Omega^n(X)^2$. Note that if an inverse $-[f_{i,l},g_{i,l}]$ appears in such a finite sum, we may use bilinearity to replace it, up to homotopy, with $[f_{i,l},-g_{i,l}]$ where $(f_{i,l},-g_{i,l})$ remains wedge-disjoint. Hence, the form $\sum_{l=1}^{m_i}[f_{i,l},g_{i,l}]$ is justified. Since it is not assumed that $\{\sum_{l=1}^{m_i}[f_{i,l},g_{i,l}]\}_{j\in\bbn}$ converges to the basepoint, we can not simply form an infinite sum from this sequence. Let $p_{\geq k}:X\to X_{\geq k}$ denote the canonical retraction map for each $k\in\bbn$. Then \[w_i=p_{\geq k_i}\circ w_i\simeq p_{\geq k_i}\circ\sum_{l=1}^{m_i}[f_{i,l},g_{i,l}]=\sum_{l=1}^{m_i}[p_{\geq k_i}\circ f_{i,l},p_{\geq k_i}\circ g_{i,l}]\] in $X_{\geq k_i}$ where each pair $(p_{\geq k_i}\circ f_{i,l},p_{\geq k_i}\circ g_{i,l})$ is still wedge-disjoint. Thus, \[\left\{\sum_{l=1}^{m_i}[p_{\geq k_i}\circ f_{i,l},p_{\geq k_i}\circ g_{i,l}]\right\}_{i\in\bbn}\] converges to the basepoint and is sequentially homotopic to $\{w_i\}_{i\in\bbn}$. An infinite horizontal concatenation gives \[\sum_{i=1}^{\infty}w_i\simeq \sum_{i=1}^{\infty}\sum_{l=1}^{m_i}[p_{\geq k_i}\circ f_{i,l},p_{\geq k_i}\circ g_{i,l}].\] Using a reparamterization homotopy to write the double sum as a single sum indexed by $\bbn$, we conclude that a generic element of $\mcw_{2n-1}(X)$ may be expressed in the form $\left[\sum_{j=1}^{\infty}[f_i',g_i']\right]$ for $\{f_i'\}_{i\in\bbn},\{g_i'\}_{i\in\bbn}\in\loopn_{\bbn}(X)$ and where each pair $(f_i',g_i')$ is wedge-disjoint. This completes the proof.
\end{proof}

\begin{lemma}\label{closurelemma2}
Let $X=\sw_{j\in\bbn}X_j$ be a shrinking wedge of CW-complexes and $n\geq 2$. The subgroup $\mathcal{W}_{2n-1}(X)$ is closed under infinite sums. Moreover, \[\mathcal{W}_{2n-1}(X)=cl_{\Sigma}\left(\{[\alpha,\beta]\in \pi_{2n-1}(X)\mid (\alpha,\beta)\in \pi_{n}(X)^2\text{ is wedge-disjoint}\}\right).\]
\end{lemma}

\begin{proof}
Define $S$ as in the proof of Lemma \ref{closurelemma} so that $\mcw_{2n-1}(X)=\llbracket S\rrbracket$. By Proposition \ref{inclusionprop}, we have $S\subseteq \llb S\rrb\leq cl_{\Sigma}(S)$. Thus, to prove $\llb S\rrb= cl_{\Sigma}(S)$, it suffices to show that $\llb S\rrb$ is closed under infinite sums. Suppose $\{h_i\}_{i\in\bbn}\in \loopn_{\bbn}(X)$ where $[h_i]\in \llb S\rrb$ for each $i\in\bbn$. By making the replacement from Corollary \ref{replacementcorollary} (without changing sequentially homotopy class of $\{h_i\}_{i\in\bbn}$), we may assume there exists a sequence of integers $\{k_{i}\}_{i\in\bbn}\to\infty$ such that $\im(h_i)\subseteq X_{\geq k_{i}}$ for each $i\in\bbn$. By Lemma \ref{closurelemma}, we have $[h_i]=\left[\sum_{j=1}^{\infty}[f_{i,j},g_{i,j}]\right]$ where $\{f_{i,j}\}_{j\in\bbn},\{g_{i,j}\}_{j\in\bbn}\in\loopn_{\bbn}(X)$ and where each pair $(f_{i,j},g_{i,j})$ is wedge-disjoint. Then \[h_i=p_{\geq k_i}\circ h_i\simeq p_{\geq k_i}\circ \sum_{j=1}^{\infty}[f_j,g_j]=\sum_{j=1}^{\infty}[p_{\geq k_i}\circ f_{i,j},p_{\geq k_i}\circ g_{i,j}]\] in $X_{\geq k_{i}}$ where each $(p_{\geq k_i}\circ f_{i,j},p_{\geq k_i}\circ g_{i,j})$ is a wedge-disjoint pair. This defines a sequential homotopy $\{h_i\}_{i\in\bbn}\simeq \left\{\sum_{j=1}^{\infty}[p_{\geq k_i}\circ f_{i,j},p_{\geq k_i}\circ g_{i,j}]\right\}_{i\in\bbn}$ and thus \[\sum_{i=1}^{\infty}h_i\simeq \sum_{i=1}^{\infty}\sum_{j=1}^{\infty}[p_{\geq k_i}\circ f_{i,j},p_{\geq k_i}\circ g_{i,j}].\] Using an infinite commutativity homotopy from Corollary \ref{infcomm2}, we may re-index the double infinite sum (indexed by $\bbn^2$) to be a single sum indexed by $\bbn$. We conclude that $\left[\sum_{i=1}^{\infty}h_i\right]\in \llb S\rrb$.
\end{proof}

\begin{corollary}\label{inkercor}
If $X=\sw_{j\in\bbn}X_j$ is a shrinking wedge of CW-complexes and $n\geq 2$, then $\mcw_{2n-1}(X)\leq \ker(\sigma_{\#})$ as subgroups of $\pi_{2n-1}(X)$.
\end{corollary}

\begin{proof}
If the pair $(\alpha,\beta)\in \pi_{n}(X)^2$ is wedge-disjoint and $j\in\bbn$ is fixed, then each projection $p_j:X\mapsto X_j$ trivializes at least one of $\alpha$ or $\beta$. Thus $p_{j\#}([\alpha,\beta])=0\in\pi_{2n-1}(X_j)$ for each $j$. It follows that $[\alpha,\beta]\in \ker(\sigma_{\#})$. Thus, if $S=\{[\alpha,\beta]\in \pi_{2n-1}(X)\mid (\alpha,\beta)\in \pi_{n}(X)^2\text{ is wedge-disjoint}\}$, then $S\subseteq \ker(\sigma_{\#})\leq \pi_{2n-1}(X)$. Since $\ker(\sigma_{\#})$ is closed under infinite sums by Lemma \ref{sigmaclosedlemma} and $\mcw_{2n-1}(X)=cl_{\Sigma}(S)$ by Lemma \ref{closurelemma2}, we have $\mcw_{2n-1}(X)= cl_{\Sigma}(S)\leq cl_{\Sigma}(\ker(\sigma_{\#}))=\ker(\sigma_{\#})$.
\end{proof}

\subsection{Characterizing Whitehead Products in $\pi_{2n-1}(\sw_{j\in\bbn}X_j)$}

In this subsection, we consider a shrinking wedge $X=\sw_{j\in\bbn}X_j$ of $(n-1)$-connectex CW-complexes $X_j$ with basepoint $x_0$. For each $k\in\bbn$, the inclusion maps $X_{k}\to X_{\geq k}$ and $X_{\geq k+1}\to X_{\geq k}$ and the bilinearity of the Whitehead bracket combine to give a homomorphism $\phi_{k}:\pi_{n}(X_k)\otimes \pi_{n}(X_{>k})\to \pi_{2n-1}(X_{\geq k})$ defined by $\phi_{k}(\alpha\otimes \beta)= [\alpha,\beta]$ on simple tensors. Any such pair $(\alpha,\beta)$ is wedge-disjoint and so $\im(\phi_k)\leq \mcw_{2n-1}(X)$. 

\begin{definition}\label{phidef}
Let \[\Phi:\prod_{k\in\bbn}\pi_{n}(X_k)\otimes \pi_{n}(X_{>k})\to  \pi_{2n-1}(X)\] be the canonical homomorphism defined by $\Phi=\mathbf{S}_{2n-1}\circ \prod_{k\in\bbn}\phi_k$ where $\mathbf{S}_{2n-1}$ is the summation homomorphism from Proposition \ref{sumhomomoprop} and $\phi_k$ is the homomorphism defined in the previous paragraph.
\end{definition}

\begin{remark}\label{imageinremark}
In its most general formulation, one can evaluate $\Phi$ as follows. Given an element $(g_k)_{k\in\bbn}\in\prod_{k\in\bbn}\pi_{n}(X_k)\otimes \pi_{n}(X_{>k})$ where $g_k=\sum_{i=1}^{r_k}[f_{k,i}]\otimes [g_{k,i}]$ for $f_{k,i}\in \Omega^n(X_k)$ and $g_{k,i}\in \Omega^n(X_{>k})$, $\Phi((g_k)_{k\in\bbn})$ is the element $\left[\sum_{k=1}^{\infty}\sum_{i=1}^{r_k}[f_{k,i},g_{k,i}]\right]$ of $\pi_{2n-1}\left(\sw_{k\in\bbn}X_k\right)$. Since each pair $([f_{k,i}],[g_{k,i}])$ is a wedge-disjoint pair, the Whitehead product $[f_{k,i},g_{k,i}]$ is an element of $\mcw_{2n-1}(X)$. Moreover, since $\Phi((g_k)_{k\in\bbn})$ is an infinite sum of such Whitehead products, we have $\im(\Phi)\leq \mcw_{2n-1}(X)$.
\end{remark}

The following Lemma is proved using standard methods including the Hurewicz and Kunneth Theorems.

\begin{lemma}\cite[\textsection XI.1(1.7)]{WhiteheadEOH}\label{basecaselemma}
Let $A$ and $B$ be $(n-1)$-connected CW-complexes. If $\sigma:A\vee B\to A\times B$ is the canonical inclusion, then the Whitehead product operation $\tau:\pi_{n}(A)\otimes \pi_n(B)\to \pi_{2n-1}(A\vee B)$ determined by $\tau(\alpha\otimes \beta)= [\alpha,\beta]$ is an isomorphism onto the kernel of the homomorphism induced by the inclusion $A\vee B\to A\times B$.
\end{lemma}

There are various ways one could state a recursive application of the previous lemma. The following is most convenient for our purposes.

\begin{corollary}\label{finitecasecor}
Let $A_1,A_2,\dots,A_m$ be $(n-1)$-connected CW-complexes. Then the Whitehead product operation \[\tau:\bigoplus_{k=1}^{m-1}\left(\pi_{n}(A_k)\otimes \pi_n\left(\bigvee_{j=k+1}^{m} A_k\right)\right)\to \pi_{2n-1}\left(\bigvee_{j=1}^{m}A_j\right),\] determined by mapping $\alpha_k\otimes\beta_k\mapsto [\alpha_k,\beta_k]$ on each summand, is an isomorphism onto the kernel of the homomorphism induced by the inclusion $\bigvee_{j=1}^{m}A_j\to \prod_{j=1}^{m}A_j$.
\end{corollary}

Although tensor products do not preserve infinite direct products in general, the following situation is an exception. The next lemma is certainly known but does not seem to appear frequently. We include the proof for the sake of completion.

\begin{lemma}\label{tensorlemma}
If $G$ and $H_j$, $j\in J$ are finitely generated abelian groups, then the canonical homomorphism $\theta:G\otimes \prod_{j\in J}H_j\to \prod_{j\in J}G\otimes H_j$ defined by $\theta(a\otimes (b_j)_{j\in J})= (a\otimes b_j)_{j\in J}$ is an isomorphism.
\end{lemma}

\begin{proof}
Since $\theta$ is induced by the projection maps $G\otimes \prod_{j\in J}H_j\to G\otimes H_j$, it is a well-defined homomorphism. We check that $\theta$ is bijective. Since tensor products distribute over direct sums and direct products commute with finite direct sums, it suffices to verify the case where $G$ is a cyclic group. This is clear when $G=\bbz$ so we may assume $G=\bbz_{m}$ for $m>1$. In general, if $G$ is cyclic with generator $1$, then we can write every element of a tensor product $G\otimes H$ as a simple tensor $1\otimes h$. Since $\theta(1\otimes (h_j)_{j})=(1\otimes h_j)_{j}$, it follows that $\theta$ is surjective. For injectivity, note that since each $H_{j}$ is a finite direct product of cyclic groups, we may split the product $\prod_{j\in J}H_j$ as $A\oplus B$ where $A$ is a product of finite cyclic groups and $B$ is a product of infinite cyclic groups. Again applying distribution over direct sums we reduce to two cases: (1) for all $j\in J$, $H_j=\bbz_{n_j}$ for some integer $n_j>1$ and (2) for all $j\in J$, $H_j=\bbz$.\\
Case (1): Let $g_j=\gcd(m,n_j)$ and identify $\bbz_{m}\otimes\bbz_{n_j}=\bbz_{g_j}$ for each $j\in J$. Under this identification, we have $\theta(a\otimes (b_j)_{j})=(ab_j)_{j}\in \prod_{j\in J}\bbz_{g_j}$. If $\theta(1\otimes (h_j)_{j})=(h_j)_{j}=0$, then $h_j=k_jg_j$ for $k_j\in\bbz$. Write $p_jm+q_jn_j=g_j$ for $p_j,q_j\in\bbz$. Then
\[1\otimes (h_j)_{j}=1\otimes (k_jp_jm+k_jq_jn_j)_j=1\otimes (k_jp_jm)_j=m\otimes (k_jp_j)_j=0\]in $\bbz_{m}\otimes \prod_{j\in J}\bbz_{n_j}$.\\
Case (2): Identify $\bbz_{m}\otimes \bbz=\bbz_{m}$. Under this identification, we also have $\theta(a\otimes (b_j)_{j})=(ab_j)_{j}$. Suppose $\theta(1\otimes (h_j)_j)=(h_j)_j=0\in \prod_{j\in\bbn}\bbz_m$. For each $j\in J$, we have $h_j=k_jm$ for some $k_j\in \bbz$. Then $1\otimes (h_j)_j=1\otimes (k_jm)_j=m\otimes (k_j)_j=0$ in $\bbz_{m}\otimes \prod_{j\in\bbn}\bbz$.
\end{proof}

\begin{lemma}\label{projectionlemma}
Suppose $n\geq 2$, $\sw_{k\in\bbn}X_k$ is a shrinking wedge of finite $(n-1)$-connected CW-complexes, and $K\in \bbn$ is fixed. An element $g\in \pi_n(X_K)\otimes \pi_n(X_{>K})$ is non-trivial if and only if there exists $j>K$ such that if $p_{K,j}:X_{>K}\to X_j$ is the canonical retraction, then $(id\otimes p_{K,j\#})(g)$ is non-trivial in $\pi_n(X_K)\otimes \pi_n(X_j)$.
\end{lemma}

\begin{proof}
Since $X_{>K}=\sw_{i=K+1}^{\infty}X_i$ is a shrinking wedge of $(n-1)$-connected CW-complexes, Theorem \ref{connectedtheorem} gives that the canonical homomorphism $h:\pi_n(X_{>K})\to \prod_{i>K}\pi_n(X_{i})$ induced by projection maps is an isomorphism. Since each $X_k$ is a finite simply-connected CW-complex, each of the groups $\pi_n(X_k)$ is finitely generated by a Theorem of Serre \cite{Serre}. Since Lemma \ref{tensorlemma} applies, the canonical homomorphism $\theta:\pi_n(X_K)\otimes \prod_{i>K}\pi_n(X_i) \to\prod_{i>K}\pi_n(X_K)\otimes\pi_n(X_i)$ is an isomorphism. Thus, for each $j>K$, we have the following commutative diagram where the top arrows are the indicated isomorphisms and $r_j$ is the $j$-th product-projection map.
\[\xymatrix{
\pi_n(X_K)\otimes \pi_n(X_{>K}) \ar[dr]_-{id\otimes p_{K,j\#}} \ar[r]^-{id\otimes h} & \pi_n(X_K)\otimes \prod_{i>K}\pi_n(X_i) \ar[r]^-{\theta} &  \prod_{i>K}\pi_n(X_K)\otimes\pi_n(X_i) \ar[dl]^-{r_j} \\
& \pi_n(X_K)\otimes\pi_n(X_j)
}\]
Allowing $j$ to vary, note that an element $g'$ of the product $\prod_{i>K}\pi_n(X_K)\otimes\pi_n(X_i)$ is trivial if and only if it's projection $r_j(g')$ is trivial for all $j>K$. The lemma follows.
\end{proof}

We now prove a part of Theorem \ref{maintheorem}.

\begin{lemma}\label{mainlemma1}
If $n\geq 2$ and $X=\sw_{k\in\bbn}X_k$ is a shrinking wedge of finite $(n-1)$-connected CW-complexes, then the canonical homomorphism \[\Phi:\prod_{k=1}^{\infty}\left(\pi_{n}(X_k)\otimes \pi_{n}(X_{>k})\right)\to \pi_{2n-1}\left(X\right)\] defined in Definition \ref{phidef} is injective.
\end{lemma}

\begin{proof}
Suppose $0\neq g=(g_k)_{k\in\bbn}\in \prod_{k=1}^{\infty}\left(\pi_{n}(X_k)\otimes \pi_{n}(X_{>k})\right)$. Then there exists $K\in\bbn$ such that if $R_K$ is the $K$-th projection map of the product, then $R_K(g)=g_K\neq 0$. By Lemma \ref{projectionlemma}, there exists $j>K$ such that if $p_{K,j}:X_{>k}\to X_j$ is the projection map, then $(id \otimes p_{K,j\#})(g_K)\neq 0$ in $\pi_{n}(X_K)\otimes \pi_n(X_j)$. There is also a canonical retraction  $Q_{K,j}:X\to X_K\vee X_j$. We have the following commuting diagram where $\tau$ is the injective homomorphism from Lemma \ref{basecaselemma}.
\[\xymatrix{
\prod_{k=1}^{\infty}\left(\pi_{n}(X_k)\otimes \pi_{n}(X_{>k})\right) \ar[d]_-{(id\otimes p_{K,j\#})\circ R_K} \ar[r]^-{\Phi}  & \pi_{2n-1}\left(X\right) \ar[d]^-{Q_{K,j\#}} \\
\pi_{n}(X_K)\otimes \pi_n(X_j) \ar[r]_-{\tau} & \pi_{2n-1}(X_K\vee X_j)
}\]
Since $\tau$ is injective, $Q_{K,j\#}\circ \Phi(g)=\tau\circ(id \otimes p_{K,j\#})(g_K)\neq 0$. It follows that $\Phi(g)\neq 0$.
\end{proof}

\begin{lemma}\label{mainlemma2}
If $n\geq 2$ and $X=\sw_{j\in\bbn}X_j$ is a shrinking wedge of finite $(n-1)$-connected CW-complexes, then $\mcw_{2n-1}(X)$ is the image of the canonical homomorphism $\Phi:\prod_{k=1}^{\infty}\left(\pi_{n}(X_k)\otimes \pi_{n}(X_{>k})\right)\to \pi_{2n-1}\left(X\right)$.
\end{lemma}

\begin{proof}
In Remark \ref{imageinremark} it is shown that $ \im(\Phi)\leq \mcw_{2n-1}(X)$. For the other inclusion, let $\alpha\in\mcw_{2n-1}(X)$. By Lemma \ref{closurelemma}, we may write $\alpha=\left[\sum_{j=1}^{\infty}[f_j,g_j]\right]$ where $\{f_j\}_{j\in\bbn},\{g_j\}_{j\in\bbn}\in\loopn_{\bbn}(X)$ and where each pair $(f_j,g_j)$ is wedge-disjoint. Applying Lemma \ref{replacementlemma} to both $\{f_j\}_{j\in\bbn}$ and $\{g_j\}_{j\in\bbn}$, we may assume that there is a sequence of integers $\{m_j\}_{j\in\bbn}\to\infty$ such that $\im(f_j)\cup\im(g_j)\subseteq X_{\geq m_j}$ for each $j$. Recall from the proof of Lemma \ref{replacementlemma} that such a replacement for $f_j$ has image contained in $\im(f_j)$ (and similarly for $g_j$). Thus each pair $(f_j,g_j)$ will still be wedge-disjoint following this replacement. Corollary \ref{irwhiteheadcor} ensures this replacement does not change the homotopy class of the sum $\sum_{j=1}^{\infty}[f_j,g_j]$.

Recall from Theorem \ref{connectedtheorem} that the homomorphism $h:\pi_{n}(X)\to \prod_{k\in\bbn}\pi_n(X_k)$ induced by the projection maps is an isomorphism. This allow us to write $f_j\simeq \sum_{i=1}^{\infty}f_{j,i}$ and $g_j\simeq\sum_{i=1}^{\infty}g_{j,i}$ for unique homotopy classes $[f_{j,i}],[g_{j,i}]\in \pi_n(X_i)$ (and we assume $f_{j,i},g_{j,i}\in \Omega^n(X_i)$). In particular, if $p_i:X\to X_{i}$ is the canonical retraction, then $p_{i\#}([f_j])=[f_{j,i}]$ and $p_{i\#}([g_j])=[g_{j,i}]$. Moreover, whenever $[f_{j,i}]=0$, we take $f_{j,i}$ to be constant and whenever $[g_{j,i}]=0$, we take $g_{j,i}$ to be constant. Thus, for each $j$, the sums $\sum_{i=1}^{\infty}f_{j,i}$ and $\sum_{i=1}^{\infty}g_{j,i}$ have image in $X_{\geq m_j}$. Since $\{m_j\}_{j\in\bbn}\to\infty$, we have sequential homotopies $\{f_j\}_{j\in\bbn}\simeq \{\sum_{i=1}^{\infty}f_{j,i}\}_{j\in\bbn}$ and $\{g_j\}_{j\in\bbn}\simeq \{\sum_{i=1}^{\infty}g_{j,i}\}_{j\in\bbn}$. By Corollary \ref{irwhiteheadcor}, there is a sequential homotopy $\{[f_j,g_j]\}_{j\in\bbn}\simeq\{[\sum_{i=1}^{\infty}f_{j,i},\sum_{i=1}^{\infty}g_{j,i}]\}_{j\in\bbn}$. A horizontal concatenation then gives \[\alpha=\left[\sum_{j=1}^{\infty}[f_j,g_j]\right]=\left[\sum_{j=1}^{\infty}\left[\sum_{i=1}^{\infty}f_{j,i}, \sum_{i=1}^{\infty}g_{j,i}\right]\right].\]
Applying Theorem \ref{mainidentitylemma} to this last expression for $\alpha$, we have
\begin{equation}\label{alphaeqn}
\alpha=\left[\sum_{(j,i)\in\bbn^2}[f_{j,i},g_{j,i}]+
\sum_{(j,i)\in\bbn^2}\left[f_{j,i},\sum_{k>i}g_{j,k}\right]+(-1)^{n^2}\sum_{(j,i)\in\bbn^2}\left[g_{j,i},\sum_{k>i}f_{j,k}\right]\right].
\tag{$\ast$}
\end{equation}
Note that for fixed $j\in\bbn$, the pair $(f_j,g_j)$ is a wedge-disjoint pair and thus for each $k\geq m_j$, at least one of the maps $f_{j,k}$ or $g_{j,k}$ must be constant. Hence, for any pair $(j,i)$, at least one of $f_{j,i}$ or $g_{j,i}$ is constant. It follows that we may delete, using a sequential null-homotopy, all Whitehead products in \eqref{alphaeqn} that have a constant map in one of its two coordinates. In particular, this deletes the left-most term $\sum_{(j,i)\in\bbn^2}[f_{j,i},g_{j,i}]$ entirely. For each $i\in\bbn$, the sets $A(i)=\{j\in\bbn\mid f_{j,i}\text{ is not constant}\}$ and $B(i)=\{j\in\bbn\mid g_{j,i}\text{ is not constant}\}$ are finite. Commuting the sums indexed by $i$ and $j$ (Lemma \ref{doublesumlemma}), we have
\[\alpha=\left[\sum_{i=1}^{\infty}\sum_{j\in A(i)}\left[f_{j,i},\sum_{k>i}g_{j,k}\right]+ (-1)^{n^2}\sum_{i=1}^{\infty}\sum_{j\in A(i)}\left[g_{j,i},\sum_{k>i}f_{j,k}\right]\right].\]
Note that $\gamma_i=\sum_{j\in A(i)}[f_{j,i}]\otimes [\sum_{k>i}g_{j,k}]$ and $\delta_i=\sum_{j\in B(i)}[g_{j,i}]\otimes [\sum_{k>i}f_{j,k}]$ are both elements of $\pi_n(X_i)\otimes \pi_{n}(X_{>i})$. The formula for $\Phi$ gives
\[\Phi((\gamma_i)_{i})=\left[\sum_{i=1}^{\infty}\sum_{j\in A(i)}\left[f_{j,i},\sum_{k>i}g_{j,k}\right]\right]\]and 
\[\Phi((\delta_i)_{i})=\left[\sum_{i=1}^{\infty}\sum_{j\in A(i)}\left[g_{j,i},\sum_{k>i}f_{j,k}\right]\right].\]
We conclude that $\Phi((\gamma_i)_{i}+(-1)^{n^2}(\delta_i)_{i})=\alpha$, which proves $\alpha\in \im(\Phi)$.
\end{proof}

\begin{proof}[Proof of Theorem \ref{maintheorem}]
Lemma \ref{mainlemma1} shows that $\Phi$ is injective and Lemma \ref{mainlemma2} gives $\im(\Phi)=\mcw_{2n-1}(X)$. For the last statement, we have $\pi_n(X_{>j})\cong \prod_{k>j}\pi_{n}(X_j)$ for each $j\in\bbn$ by Theorem \ref{connectedtheorem}. An appropriate combination of these isomorphisms with $\Phi$ completes the proof.
\end{proof}

Once again applying the fact that finite simply connected CW-complexes have finitely generated homotopy groups, we may combine Theorem \ref{maintheorem} with Lemma \ref{tensorlemma} to obtain the following.

\begin{corollary}\label{maincor}
Suppose $n\geq 2$ and $X=\sw_{j\in\bbn}X_j$ is a shrinking wedge of finite $(n-1)$-connected CW-complexes. Then the subgroup $\mcw_{2n-1}(X)$ of $\pi_{2n-1}(X)$ is canonically isomorphic to $\prod_{k>j}\pi_{n}(X_j)\otimes \pi_{n}(X_k)$.
\end{corollary}

\begin{remark}
One may be tempted to interpret the inverse of the isomorphism from Corollary \ref{maincor} as a map $\prod_{k>j}\pi_{n}(X_j)\otimes \pi_{n}(X_k)\to \mcw_{2n-1}(X)$ sending $([f_j]\otimes [g_k])_{k>j}$ to the homotopy class infinite sum $\sum_{k>j}[f_j,g_k]$. However, the map $\sum_{k>j}[f_j,g_k]$ will not typically be continuous for if all of the maps $f_j$ and $g_j$ are not null-homotopic, then the collection $\{[f_j,g_k]\}_{k>j}$ does not cluster at the basepoint, e.g. $\{[f_1,g_k]\}_{k\geq 2}$ will not converge to $x_0$. Hence, it seems that our application of the purely algebraic result Lemma \ref{tensorlemma} tends to forget information about representing maps. From a topological perspective, the isomorphism in Theorem \ref{maintheorem} may be more practical.
\end{remark}

\end{document}